\documentclass[a4paper]{amsart}
\usepackage{amsfonts,amsmath,amssymb,amsthm,url,mathtools}
\usepackage[noadjust]{cite}
\usepackage{epsfig}
\usepackage[british]{babel}
\usepackage[utf8]{inputenc}
\usepackage{hyperref, enumerate}

\numberwithin{equation}{section}

\newtheorem{theorem}{Theorem}[section]
\theoremstyle{definition}
\newtheorem{example}[theorem]{Example}
\newtheorem{algorithm}[theorem]{Algorithm}
\newtheorem{proposition}[theorem]{Proposition}
\newtheorem{lemma}[theorem]{Lemma}
\newtheorem{corollary}[theorem]{Corollary}
\newtheorem{remark}[theorem]{Remark}
\newtheorem{definition}[theorem]{Definition}
\newcommand{\trace}{\mathrm{Tr}}
\newcommand{\HH}{\mathbf{H}}

\newcommand{\CC}{\mathbf{C}}
\newcommand{\ZZ}{\mathbf{Z}}
\newcommand{\RR}{\mathbf{R}}
\newcommand{\QQ}{\mathbf{Q}}

\newcommand{\QQhat}{\widehat{\QQ}}
\newcommand{\ZZhat}{\widehat{\ZZ}}
\newcommand{\FF}{\mathbf{F}}

\newcommand{\Gal}{\mathrm{Gal}}
\newcommand{\ag}{\mathfrak{a}}
\newcommand{\bg}{\mathfrak{b}}

\newcommand{\GL}{\mathrm{GL}}
\newcommand{\ord}{\mathrm{ord}}

\renewcommand{\AA}{\mathbf{A}}

\newcommand{\Gm}{\mathbf{G}_{\mathrm{m}}}
\newcommand{\GSpspecific}[1]{\mathrm{GSp}_{#1}}
\newcommand{\GSp}{\GSpspecific{2g}}
\newcommand{\Spspecific}[1]{\mathrm{Sp}_{#1}}
\newcommand{\Sp}{\Spspecific{2g}}
\newcommand{\FcalN}{\mathcal{F}_N}
\newcommand{\Fcalall}{\mathcal{F}_\infty}
\newcommand{\Fcal}[1]{\mathcal{F}_{#1}}

\newcommand{\Stab}{\mathrm{Stab}}
\newcommand{\units}{\times}
\newcommand{\End}{\mathrm{End}}

\newcommand{\SL}{\mathrm{SL}}
\newcommand{\matrixring}[2]{#2^{#1\times #1}}
\newcommand{\NPhirO}{N_{\reflextype,\mathcal{O}}}

\newcommand{\matrixofShimurapdfstring}[1]{the matrix}

\newcommand{\matrixofShimura}[1]{\epsilon(N_{\reflextype}(#1))}

\newcommand{\IKr}[1]{I(#1)}
\newcommand{\PKr}[1]{P(#1)}
\newcommand{\cmgp}[1]{H_{\Phi,\mathcal{O}}(#1)}
\newcommand{\cmext}[1]{\mathcal{H}(#1)}
\newcommand{\reflextype}{\Phi^{\mathrm{r}}}
\newcommand{\reflexfield}{K^{\mathrm{r}}}
\newcommand{\realreflex}{\reflexfield_0}
\newcommand{\reflexunits}{K^{\mathrm{r}\units}}
\newcommand{\reflexab}{K^{\mathrm{r}}_{\mathrm{ab}}}
\newcommand{\reflexid}{K^{\mathrm{r}\units}_{\AA}}
\newcommand{\Kid}{K^\times_\AA}

\renewcommand{\mod}{\ \mathrm{mod}\ }
\newcommand{\modstar}{\ \mathrm{mod}^\units\ }
\newcommand{\tbt}[1]{\left(\begin{array}{cc} #1 \end{array}\right)}

\newcommand{\tbtinl}[4]{\begin{psmallmatrix} #1 & #2 \\ #3 & #4\end{psmallmatrix}}
\newcommand{\OKr}{\mathcal{O}_{\reflexfield}}
\newcommand{\transpose}{^{\mathrm{t}}}
\newcommand{\wascalledg}{r}
\newcommand{\wascalledS}{T}

\newcommand{\leftmult}[3]{[#1]^{#2}_{#3}}
\newcommand{\leftchange}[2]{\leftmult{1}{#1}{#2}}
\newcommand{\wasepsilon}[1]{\leftmult{#1}{B}{B}}

\title[Explict Shimura reciprocity for Siegel modular functions]{An explicit version of Shimura's reciprocity law for Siegel modular functions}

\author{Marco Streng}
\thanks{Mathematisch Instituut, Universiteit Leiden,
P.O.~box 9512,
2300 RA Leiden,
The Netherlands. Email: \url{streng@math.leidenuniv.nl}.
During parts of the period in which this work was done, the author
was supported by EPSRC grant number
EP/G004870/1 and NWO Vernieuwingsimpuls. The author would like to thank Jared Asuncion,
Gaetan Bisson, Andreas Enge, Jean-Pierre Flori, Math\'e Hertogh, Marc Masdeu, Peter Stevenhagen, and Tonghai Yang for many useful
comments for the improvement of the exposition and the software.}

\begin{document}

 \begin{abstract}
	\noindent We give an explicit version of Shimura's
	reciprocity law for singular values
	of Siegel modular functions.
	We use this to construct the first examples of
	class invariants of quartic CM fields
	that are smaller than Igusa invariants.
	Our version also enabled a new proof of Shimura's
	reciprocity law by Tonghai Yang.
\end{abstract}
 
\maketitle

 \section{Introduction}

The values of the
modular function $j$
in imaginary quadratic numbers~$\tau$
generate abelian extensions
of imaginary quadratic fields~$K=\QQ(\tau)$.
These values $j(\tau)$ enable explicit computation
of the Hilbert class field of~$K$
and of
elliptic curves over finite fields
with a prescribed number of points
(the ``CM method'')
for primality testing and cryptography.

However, these algebraic numbers $j(\tau)$
have very large height, which limits
their usefulness in such applications.
So we consider
other modular functions~$f$ instead,
whose values are again abelian over~$K$,
hoping to find numbers of smaller height.
If these values $f(\tau)$ lie in the
same field as $j(\tau)$,
then we call them \emph{class invariants},
and they can
take the place of $j(\tau)$ in applications,
which leads to great speed-ups~\cite{enge-cm}.

The values $f(\tau)$
are acted upon by ideals (and id\`eles) of~$K$
via the Artin isomorphism.
Shimura's reciprocity law expresses this action
in terms of an action on the modular
functions~$f$ themselves,
and an explicit version of this reciprocity
law~\cite{MR563924,gee-stevenhagen}
allows one to search for class invariants in a systematic way.

There exists a higher-dimensional CM method,
with applications in hyperelliptic curve cryptography
and a more general analytic construction of class fields~\cite{spallek,HEHCC}.
A significant speedup will be obtained by replacing 
the \emph{Igusa invariants} in this construction
by \emph{smaller} class invariants.

Shimura gave various higher-dimensional analogues of his reciprocity
law~\cite{shimura-models-I, shimura-models-II, shimura-arithmetic, shimura-fourier, shimura-theta-cm, shimura-reciprocity-theta}.
Our main result (Theorems \ref{thm:general}, \ref{thm:idealgroup}, and~\ref{thm:special} below)
is a new and explicit version, suitable
for finding class invariants in the higher-dimensional setting.

We use our explicit formulation of Shimura's reciprocity law
to find the first examples of small class
invariants of quartic CM fields
(Section~\ref{sec:examples}).
Our formulation of Shimura's reciprocity law also 
inspired a new proof of Shimura's reciprocity law by
Tonghai Yang~\cite[Section~4, see also Acknowledgements]{yang-shimura}.
As a third application, Andreas Enge and the author~\cite{enge-streng}
use the explicit reciprocity law
for generalizing Schertz's work on class invariants~\cite{schertz}
to higher dimension.

\subsection{Summary of results}

Let $\mathcal{F}_N$ be the field of Siegel modular functions
of level $N$ over $\QQ(\zeta_N)$ (c.f.~\eqref{eq:FN}).
Let $f\in\mathcal{F}_N$ be such a function.
Let $\tau\in\CC^{g\times g}$ be a symmetric matrix with positive
definite imaginary part (that is, a point in the Siegel upper half space $\HH_g$).
If $\tau$ is a primitive CM point (Section~\ref{ssec:cm}),
then $f(\tau)$ is an algebraic number
and is in fact abelian over a field known as the \emph{reflex
field} $\reflexfield$ of $\tau$ (Section~\ref{sec:reflex}).

Now given $\sigma\in\mathrm{Gal}(\overline{\QQ}/\reflexfield)^{\mathrm{ab}}$,
there are various reasons why we 
would like to be able to compute $f(\tau)^\sigma$.
For example, it allows us to decide whether $f(\tau)$
is in certain subfields of $\overline{\QQ}$ and
to find its minimal polynomial over~$\reflexfield$.
This minimal polynomial can be used to speed up
explicit class field theory and explicit
CM constructions of curves and Jacobians
\cite{enge-morain, sutherlandclassinv}.

Shimura's reciprocity law~\cite{shimura-models-I, shimura-models-II, shimura-arithmetic, shimura-fourier, shimura-theta-cm, shimura-reciprocity-theta}
expresses $f(\tau)^\sigma$
in the form $F(\tau)$ where $F$ is obtained
from $f$ and $\sigma$.
The function $F$ is obtained in terms
of an action of an uncountable ad{\`e}lic group,
which is not very helpful in computation.
So in order to use such actions, one needs to
approximate the ad{\`e}lic group elements by products
of elements in particular subgroups.
We did this, and the result is an explicit reciprocity law
in terms of ideals and ray class groups, rather than
id{\`e}le class groups.

Let $\mathfrak{a}$ be an ideal.
Then Theorem~\ref{thm:general} gives (in terms of $\mathfrak{a}$ and~$\tau$)
efficiently computable
$U\in \GSp(\ZZ/N\ZZ)$ and $\tau'\in\HH_g$
with
\begin{equation}\label{eq:action}
f(\tau)^{[\mathfrak{a}]} = f^U(\tau').
\end{equation}

In turn, the action of $U$
on $\mathcal{F}_N$ can be computed
in one of the various practical ways explained in Section~\ref{sec:actioncompute}.
Moreover, we can make sure
that $\tau'$ is in a fundamental region (Section~\ref{sec:algperiodmatrices}),
allowing
for efficient numerical evaluation of $f^U(\tau')$.

We use this reciprocity law to prove (Theorem~\ref{thm:idealgroup})
a formula for the ideal group corresponding to the abelian extension
\[\cmext{N} = \reflexfield(f(\tau) : f\in\FcalN)\qquad\mbox{of}\qquad \reflexfield.\]

Computations with $f(\tau)$ become even more efficient when it is
real instead of complex. 
Proposition~\ref{prop:complexconjugation}
gives a sufficient condition
for this to happen.

The author has
implemented the actions
in SageMath~\cite{sage}
(which uses PARI~\cite{pari})
and made
the program available online at~\cite{cmcode}.

\subsection{Overview of content}

Section~\ref{sec:statement} states the results
and Sections
\ref{sec:ad}--\ref{sec:conjugationproof} contain a proof.
The action of~$U$ in~\eqref{eq:action}
becomes most explicit when expressing the function~$f$
in terms of theta constants, see Section~\ref{sec:theta}.

Section~\ref{sec:examples} gives a detailed example of how to obtain useful class invariants.

Finally, Section~\ref{sec:applications}
gives applications to computational
class field theory and to the construction
of curves over finite fields.
The final three sections (\ref{sec:theta}--\ref{sec:applications})
can be read
independently of Sections \ref{sec:ad}--\ref{sec:conjugationproof}.

 \section{Definitions and statement of the main results}\label{sec:statement}

\subsection{The upper half space} \label{ssec:results1}

 Fix a positive integer~$g$.
 The \emph{Siegel upper half space}
 $\HH_g$ is the set of~$g\times g$ symmetric
 complex 
 matrices with positive definite imaginary part.
 It parametrizes $g$-dimensional principally
 polarized abelian varieties $A$ over $\CC$
 together with a \emph{symplectic} basis $b_1,\ldots,b_{2g}$
 of their first homology.

In more detail, every abelian variety over $\CC$ is of the
form $A = \CC^g / \Lambda$
for a lattice $\Lambda$ of rank~$2g$.
A polarization is given by 
a Riemann form, i.e.,
an $\RR$-bilinear form $E$ on $\CC^g$ that restricts
to an alternating bilinear form $\Lambda\times \Lambda\rightarrow \ZZ$ 
such that~$(u,v)\mapsto E(iu,v)$ is symmetric and positive definite.
Given a $\ZZ$-basis of~$\Lambda$, there is a $2g\times 2g$ matrix, which by abuse of notation
we also denote by~$E$, such that~$E(u,v) = u\transpose E v$.
We say that~$E$ is \emph{principal} if it has determinant~$1$.
In that case, there exists a \emph{symplectic basis}, i.e., a basis such that
$E$ is given in terms of~$(g\times g)$-blocks as 
\[E = \Omega := \left(\begin{array}{rr} 0 & 1\\ -1 & 0\end{array}\right).\]
To a point $\tau\in\HH_g$, we associate the principally polarized abelian variety
with $\Lambda = \tau\ZZ^g + \ZZ^g$ and symplectic basis $\tau e_1,\ldots,\tau e_g, e_1,\ldots, e_g$,
where $e_i$ is the $i$-th standard basis element of~$\ZZ^g$.
Conversely, given a principally polarized abelian variety and a symplectic basis,
we can apply a $\CC$-linear transformation of~$\CC^g$ to write it in this
form (\cite[Chapter~8]{birkenhake-lange}).

\subsection{The algebraic groups}

Given a commutative ring~$R$, let
 $$\GSp(R) = \{ A\in\matrixring{2g}{R} : A\transpose \Omega A = \nu\Omega\text{ with $\nu\in R^\units$}\}.$$
Note that~$\nu$ defines a homomorphism of algebraic groups $\GSp\rightarrow \Gm$,
and denote its kernel by~$\Sp$.
For~$g=1$, we have simply $\GSpspecific{2}=\GL_2$, $\nu=\det$, $\Spspecific{2}=\SL_2$.

The homomorphism $\nu$ has a section~$i$,
satisfying $\nu\circ  i =\mathrm{id}_{\Gm}$, given by\footnote{Warning: our $i$
differs from Shimura's $\iota$ in the sense that $i(t)=\iota(t)^{-1}$.
We made our choice in such a way that $\iota$ is
a section of~$\nu$, where $\nu$ generalizes the determinant.}
$$i(t) = \tbt{1 & 0 \\ 0 & t}.$$
For any ring $R$ for which this makes sense, we also define
$$\GSp(R)^+ = \{A\in\GSp(R) : \nu(A)>0\}.$$
The group $\GSp(\RR)^+$
acts on $\HH_g$ by $$\tbt{a & b \\ c & d}\tau = (a \tau + b) (c \tau + d)^{-1},$$
where $a$, $b$, $c$, $d$ are $(g\times g)$-blocks.
Changes of symplectic bases correspond to the action of 
$\Sp(\ZZ)\subset \GSp(\RR)^+$ on $\HH_g$
(see Lemma~\ref{lem:mundanenew} below),
leading to the well-known fact that
$\Sp(\ZZ)\backslash \HH_g$ parametrizes
the set of isomorphism classes of principally polarized abelian varieties
of dimension~$g$.

The natural map $\Sp(\ZZ)\rightarrow \Sp(\ZZ/N\ZZ)$ is
surjective~\cite[Thm.~VII.21]{newman}.
Its kernel $\Gamma_N$ is called the \emph{principal congruence
subgroup of level~$N$}.

\subsection{Modular forms and group actions}\label{ssec:statementfirst}

A \emph{Siegel modular form} of weight $k$ and level~$N$
is a holomorphic function $f:\HH_g\rightarrow \CC$ such that for all $A=\tbtinl{a}{b}{c}{d}\in \Gamma_N$,
we have $f(A \tau) =\det (c\tau+d)^k f(\tau)$,
and which is ``holomorphic at the cusps''.
We will not define holomorphicity at the cusps,
as it is automatically satisfied for~$g>1$
by the Koecher principle~\cite{koecher},
and is a textbook condition for~$g=1$.

Every Siegel modular form $f$ has a \emph{Fourier expansion} or \emph{$q$-expansion}
\begin{equation}\label{eq:fourier}
f(\tau) = \sum_{\xi} a_{\xi} q^{\xi},
\quad a_{\xi}\in\CC,\quad q^{\xi}:=\exp (2\pi i \trace(\xi \tau) / N),
\end{equation}
where $\xi$ runs over the symmetric matrices in $\matrixring{g}{\frac{1}{2}\ZZ}$
with integral diagonal entries. 
The numbers $a_{\xi}$ are the \emph{coefficients} of the $q$-expansion.

 Let $\FcalN$
 be the field
 \begin{equation}\label{eq:FN}
 \FcalN=\left\{\frac{g_1}{g_2} : \begin{tabular}{c}
$g_i$ are Siegel modular forms of equal weight and level N,\\
with $q$-expansion coefficients in $\QQ(\zeta_N)$, and $g_2\not=0$ \end{tabular}\right\}.
\end{equation}

\begin{proposition}\label{prop:groupaction}
	There is a right action of $\GSp(\ZZ/N\ZZ)$ on $\FcalN$
	given as follows. 
	For $A\in \GSp(\ZZ/N\ZZ)$, let $t = \nu(A)$ and $B = i(t)^{-1} A$.
Then \[f^A = (f^{i(t)})^B,\] where we have:
	\begin{enumerate}\item\label{it:action1}
	For $B\in\Sp(\ZZ/N\ZZ)$, let $\widetilde{B} \in \Sp(\ZZ)$
	be such that $B = (\widetilde{B}\ \mathrm{mod}\ N)$.
	Then $f^B(\tau) = f(\widetilde{B}\tau)$ for all $f\in \FcalN$.
	\item\label{it:action2}
	For $t \in (\ZZ/N\ZZ)^{\units}$, the matrix $i(t)$ acts by the natual Galois
	action of $(\ZZ/N\ZZ)^{\units}$ on $q$-expansion coefficients,
	that is, if $$f = \frac{\sum_{\xi, k} a(\xi, k)\zeta_N^k q^\xi}{\sum_{\xi, k} b(\xi, k)\zeta_N^k q^\xi}\in\FcalN$$
	with $a(\xi,k)$, $b(\xi,k)\in\QQ$, then
	$$f^{i(t)} = \frac{\sum_{\xi,k} a(\xi, k)\zeta_{N}^{k t}q^\xi}{\sum_{\xi,k} b(\xi, k)\zeta_{N}^{k t}q^\xi} .$$
	\end{enumerate}
\end{proposition}
We give detailed references in Section~\ref{sec:ad}.

\begin{remark}
	As it is a group action, the action also satisfies $f^A = (f^{B'})^{i(t)}$
	for $B' = A i(t)^{-1} =  i(t) B  i(t)^{-1}$.
\end{remark}

\subsection{Computing the group action}\label{sec:actioncompute}

We highlight four ways in which, 
given $A\in\GSp(\ZZ/N\ZZ)$, $f\in\FcalN$, and~$\tau$,
we could compute $f^A$ or~$f^A(\tau)$.

\noindent \textbf{1. From the definition.}
The most obvious is to use \eqref{it:action1} and~\eqref{it:action2} directly.

First take $t = \nu(A)$, and write $A = i(t) B$
with $B=i(t)^{-1} A \in\Sp(\ZZ/N\ZZ)$.
Next, compute a lift $\widetilde{B}\in \Sp(\ZZ)$ of~$B$.
This can be done
by following the steps of the proof of 
\cite[Thm.~VII.21]{newman}.
Alternatively, one could compute a lift by
expressing $B$
as a product of standard generators of $\Sp(\ZZ)$
(in fact, in the case $g=1$, this results
in explicit formulas as in \cite[Lemma~6]{MR1730432}).

Then we compute $f^{i(t)}$ and evaluate it in
$\widetilde{B}\tau$.
The disadvantage of this method in practice
is that while $\tau$ can often be engineered to be
in a fundamental region where modular functions converge quickly,
we have no control over $\widetilde{B}\tau$.

For this reason, we will not take this approach,
and we promote the methods 2--4 instead.

\noindent \textbf{2. Using theta functions.}
The function 
$f\in\FcalN$ has an expression as a rational function of
of \emph{theta constants}.
If such an expression is known, then 
we can use a
direct formula
for the action of~$\GSp(\ZZ/N\ZZ)$ on $f$,
which does not even require finding a lift to $\Sp(\ZZ)$.
We give this formula in Section~\ref{sec:theta}
and use it in all our examples
in Section~\ref{sec:examples}.

\noindent \textbf{3. By selecting $f$ in such a way that the action is easy.}
It is sometimes possible to choose $f$
that are fixed by the block-lower-triangular
matrices in $\GSp(\ZZ/N\ZZ)$
and to choose $\tau$ such that $A$ is block-lower-triangular,
in which case we have $f^A = f$.
Enge and the author take this approach in~\cite{enge-streng}.

\noindent \textbf{4. Using the moduli interpretation.}
In some cases, one could use the moduli interpretation of~$f$.
We do not follow this approach in the present article,
but we do illustrate it with the following example.

\begin{example}
For $g=1$ and $N=2$, we have $\Fcal{2}=\QQ(\lambda)$,
where $\lambda$ is the Legendre invariant given
as follows.
To an $\RR$-basis $\omega_1$, $\omega_2$ of $\CC$ with $\tau = \omega_1/\omega_2\in\HH_1$,
we associate the lattice $\Lambda = \omega_1\ZZ+\omega_2\ZZ$,
the elliptic curve $E(\CC) = \CC/\Lambda$,
and the isomorphism $\phi : \FF_2^2\rightarrow E[2] : v\mapsto (\frac{1}{2}(v\cdot (\omega_1,\omega_2))\bmod\Lambda)$.
Let $e_1 = (1,0)$, $e_2 = (0,1)$, and $e_3 = (1,1)$.
For a short Weierstrass model of~$E$
with coordinates $x$ and~$y$, 
let $x_i = x(\phi(e_i))$.
Then we define
$$\lambda(\tau) = \frac{x_3-x_1}{x_2-x_1}.$$

The group $\mathrm{GSp}_{2}(\FF_2)=\SL_2(\FF_2)$
acts on $\FF_2^2$ by permutation
of $e_1$, $e_2$ and~$e_3$.
In other words, we have the isomorphism
$\sigma: \mathrm{GSp}_{2}(\ZZ/2\ZZ)\rightarrow S_3$
given by $A e_i = e_{\sigma(A)(i)}$ for $i=1,2,3$.

Writing
$\lambda_1:=0$, $\lambda_2:=1$, $\lambda_3:=\lambda\in\Fcal{2}$,
we claim that the action of Proposition~\ref{prop:groupaction} is
\begin{equation}\label{eq:lambda}\lambda^{A} = \frac{\lambda_{\rho(3)} - \lambda_{\rho(1)}}
{\lambda_{\rho(2)}-\lambda_{\rho(1)}},\quad\mbox{where}\quad\rho = \sigma(A\transpose).
\end{equation}
As special cases, we have
$$\lambda^{\begin{psmallmatrix} 1 & 1 \\ 0 & 1\end{psmallmatrix}}
= \frac{0-\lambda}{1-\lambda} = \frac{\lambda}{\lambda-1}\quad\mbox{and}\quad
\lambda^{\begin{psmallmatrix}0 & 1 \\ -1 & 0\end{psmallmatrix}} = \frac{\lambda-1}{0-1} = 1-\lambda.$$

To prove the claim, given $A$ and $\tau$, let
$\widetilde{A} = \begin{psmallmatrix} a & b \\ c & d\end{psmallmatrix}$
be a lift of $A$ and consider $(\omega_1',\omega_2') = (\omega_1,\omega_2) \widetilde{A}\transpose$
and $\tau' = \omega_1'/\omega_2' = \widetilde{A}\tau$,
which leads to $\phi' = \phi\circ A\transpose$.
Choose the Weierstrass equation of $E = E'$ in such a way that
$x_i = \lambda_i$. Then
$x_i' = x(\phi'(e_i)) = x(\phi(e_{\rho(i)})) = x_{\rho(i)}$,
hence $\lambda^A(\tau) = \lambda(\tau')$ is indeed given by \eqref{eq:lambda}.
\end{example}
With Rosenhain invariants and the appropriate
isomorphism $\mathrm{GSp}_{4}(\FF_2)\cong S_6$,
one would get the same kind of formulae for
$g=N=2$.

We hope that similar formulae can be obtained for other small
values of $g$ and $N$ on a case-by-case basis.

\subsection{Complex multiplication}\label{ssec:cm}

A \emph{primitive CM point in $\HH_g$} is a point such that
the endomorphism algebra $\End(A)\otimes\QQ$ of the
corresponding principally polarised abelian variety
$(A,E)$ is a number field $K$ of degree~$2g$.
We now explain what they look like and how to compute them.

\subsubsection{Primitive CM points}

All primitive CM points 
are of the following form. For details, see
\cite[\S I.3, 
Thms.~I.4.1, 
I.4.5
]{lang-cm}.
Let $K$ be a \emph{CM field} of degree~$2g$, that is, a totally imaginary 
quadratic extension of a totally real number field of degree~$g$.
Let~$\Phi=\{\phi_1,\ldots,\phi_g\}$
be a \emph{CM type}, that is, a set of~$g$ embeddings $K\rightarrow \CC$
such that no two are complex conjugate.
By abuse of notation, write $\Phi(x)=(\phi_1(x),\ldots,\phi_g(x))\in\CC^g$
for~$x\in K$.
Let $\mathfrak{b}$ be a lattice in~$K$, that is, a non-zero fractional ideal of an order of~$K$.
Let $\xi\in K$ be such that for all $\phi\in\Phi$, the complex number $\phi(\xi)$
lies on the positive imaginary axis,
and such that the bilinear form $E_{\xi} : K\times K\rightarrow \QQ: (x,y)\mapsto \trace(\overline{x}y\xi)$
maps $\mathfrak{b}\times\mathfrak{b}$ to~$\ZZ$.
Take $A = \CC^g/\Phi(\mathfrak{b})$ and let a polarization on $A$ be given by $E_{\xi}$ extended
$\RR$-linearly from $\mathfrak{b}$ to~$\CC^g$.
Finally, let $\mathcal{O}=\{x\in K : x\mathfrak{b}\subset\mathfrak{b}\}$ be the multiplier ring
of~$\mathfrak{b}$, and embed it into $\End(A)$ by taking $x\Phi(u) = \Phi(xu)$ and extending
this linearly.
We find an embedding $K\rightarrow \End(A)\otimes \QQ$.
Let $B$ be a symplectic basis of~$\mathfrak{b}$ for the pairing~$E_{\xi}$.
Then we get a point 
$$\tau = (\Phi(b_{g+1})|\cdots|\Phi(b_{2g}))^{-1}(\Phi(b_1)|\cdots|\Phi(b_g))\in\HH_g.$$
We denote this point also by $\tau(\Phi, \mathfrak{b}, \xi, B)$
or simply by $\tau(\Phi, B)$.
It is a primitive CM point if and only if $A$ is simple, which happens
if and only if $\Phi$ is \emph{primitive}, that is,
if and only if $\Phi_{|K'}$ is not a CM type for any CM subfield $K'\subset K$.
Moreover, all primitive CM points are of this form.

We will make the reciprocity law explicit in terms of the quadruples
$(\Phi, \mathfrak{b}, \xi, B)$.

\subsubsection{Computing the primitive CM points}
\label{sec:algperiodmatrices}

	Given $K$, we can find representatives $(\Phi, \mathfrak{b},\xi)$ for all isomorphism
	classes of principally polarized abelian varieties with CM by $\mathcal{O}_K$
	using van Wamelen's algorithm~\cite[Algorithm~1]{vanwamelen}.
	For a version of this algorithm without duplicates, see
	\cite[Algorithm~4.12]{runtime}.
	
	\label{sec:basis}
	A symplectic basis~$B$ can be computed using classical
	methods that are available as \verb!E.symplectic_form()!
	in SageMath~\cite{sage} or the
	\verb!FrobeniusFormAlternating! function in Magma~\cite{magma}.
	See also~\cite[Algorithm~5.2]{runtime}.
	
	Together, this gives a method for finding all CM points for~$\mathcal{O}_K$,
	and we implemented this as \verb!CM_Field(...).period_matrices()!
	in \cite{cmcode},
	which returns CM points in the form of SageMath objects \verb!tau!
	that include the data of $\Phi$, $\mathfrak{b}$,
	$\xi$, and the basis~$B$ and can produce arbitrary-precision approximations
	of~$\tau$.

	In practical computations, one wants to take $B$ such that numerical
	formulas for modular forms converge quickly when evaluated
	in~$\tau$.
	This can be done by first taking $B$ arbitrary and then applying
	an $\Sp(\ZZ)$-reduction algorithm to~$\tau$ to move it to
	a nice region such as a fundamental domain,
	and adjusting $B$
	accordingly, see
	\cite{DHBHS} and \cite[Section~1.3]{KLLRSS}.
	The specific case $g=1$ comes down to Gauss reduction of quadratic forms,
	and details for $g=2$ are given in Dupont's thesis~\cite{dupont}.
	We implemented this for $g\leq 2$ as \verb!tau.reduce()! in~\cite{cmcode}.
    For an implementation for $g=3$, see Kılıçer~\cite{Genusthreereduction}.

\subsection{The type norm}\label{sec:reflex}\label{sec:typenorm}

An important ingredient in the reciprocity law
is the type norm map $N_{\Phi}: K\rightarrow \CC:
x\mapsto \prod_{\phi\in\Phi} \phi(x)$
associated to a CM type~$\Phi$.
Its image generates the \emph{reflex field} $\reflexfield = \QQ(N_\Phi(K))\subset \CC$ of~$\Phi$,
and there is a reflex type norm map
\[ N_{\reflextype} : \reflexfield\rightarrow K :
x \mapsto \prod_{\psi\in\reflextype} \psi(x),\]
where the product is taken over the \emph{reflex type}
$\reflextype$, i.e., the set of embeddings $\psi : \reflexfield \rightarrow \overline{K}$
such that there is a map $\phi:\overline{K}\rightarrow\CC$ with $\phi\circ\psi =\mathrm{id}_{\reflexfield}$
and $\phi_{|K}\in\Phi$.

The reflex type norm extends to ideals via
(\cite[Proposition~29 in \S8.3]{shimura-taniyama})
\[ N_{\reflextype}(\mathfrak{a})\mathcal{O}_L =
\prod_{\psi\in\reflextype} (\psi(\mathfrak{a})\mathcal{O}_L)\]
for any number field $L\subset \overline{K}$ containing the images
$\psi(\reflexfield)$ for all $\psi\in\reflextype$.
We implemented this as \verb!Phir = Phi.reflex()!
and \verb!Phir.type_norm()! in \cite{cmcode}.

Given a positive integer $M$ and an order $A$ in a number field
such that $A$ is maximal at all primes dividing~$M$,
let $I_A(M)$ be the group of fractional ideals
of $A$ that can be written as $\ag\bg^{-1}$
with $\ag+MA = A = \bg+MA$.
We use the shorthand $$\IKr{M} = I_{\mathcal{O}_{\reflexfield}}(M)$$
and observe that $N_{\reflextype}$ sends $\IKr{M}$
to $I_{\mathcal{O}_{K}}(M)$.

In order for our results to apply to arbitrary orders
$\mathcal{O}\subset K$, we give the following
variant of~$N_{\reflextype}$.
Let $F$ be the smallest positive integer such that $F\mathcal{O}_K$
is contained in $\mathcal{O}$. Then there is a natural isomorphism
$I_{\mathcal{O}}(F)\rightarrow I_{\mathcal{O}_K}(F) :\ag \mapsto \ag\mathcal{O}_K$,
and we define $$\NPhirO: I(F)\rightarrow  I_{\mathcal{O}}(F)\quad\mbox{by}
\quad\NPhirO(\ag)\mathcal{O}_K = N_{\reflextype}(\ag).$$
In particular, we have $N_{\reflexfield, \mathcal{O}_K} = N_{\reflexfield}$.

\subsection{The first main theorem}

Given a number field $K$, two bases $B = (b_1,\ldots,b_n)$
and $C=(c_1,\ldots, c_n)$ of $K$ over $\QQ$ and an element $x\in K$,
we denote by $\leftmult{x}{C}{B}$ the $n\times n$ matrix over $\QQ$
such that for each $j$ the $j$th row is $xc_j$ expressed in terms of~$B$.
If we interpret $B$ and $C$ as column vectors in $K^n$,
then we have
\begin{equation}\label{eq:matrixofmult}
	x\ C =  \leftmult{x}{C}{B}\ B.
\end{equation}

We say that a matrix $M\in\matrixring{d}{\QQ}$
is \emph{invertible mod $N$}
if the numerator of the determinant and the denominators 
of all coefficients are coprime to~$N$.
In that case, reduction modulo $N$ defines
a matrix $(M\ \mathrm{mod}\ N)\in\GL_d(\ZZ/N\ZZ)$.
\begin{theorem}[General reciprocity law]\label{thm:general}
	Let $\tau=\tau(\Phi, \mathfrak{b}, \xi, B)\in\HH_g$ be a primitive CM point with CM field~$K$,
	let $N$ be a positive integer and let $f\in\mathcal{F}_N$
	be a function that does not have a pole at $\tau$.
	Let $F$ be the smallest positive integer such that $F\mathcal{O}_K$ is contained in the multiplier ring
	$\mathcal{O}$ of~$\mathfrak{b}$ ($F=1$ if $\mathcal{O}=\mathcal{O}_K$).
	Then $f(\tau)$ lies in the ray class field of $\reflexfield$ for the modulus $NF$.
	
	For any fractional ideal
	$\mathfrak{a}\in \IKr{NF}$, if $[\mathfrak{a}]$
	is the class of $\mathfrak{a}$ in the ray class group mod~$NF$,
	then $f(\tau)^{[\mathfrak{a}]}$ is given as follows.
	
	Choose a symplectic basis
	$C$
	of~$\NPhirO(\mathfrak{a})^{-1}\mathfrak{b}$
	with respect to $E_{
		N(\mathfrak{a})\xi}$
	and let \[\tau' = \tau(\Phi, \NPhirO(\mathfrak{a})^{-1}\mathfrak{b},
	N(\mathfrak{a})\xi, C).\]
	Then $M:=\leftchange{C}{B}$
	is in 
	$\GSp(\QQ)^{+}$,
	with $\nu(M) = N(\mathfrak{a})^{-1}$,
	and is invertible mod~$N$.
	Moreover, we have
	$U:=(M\ \mathrm{mod}\ N)^{-1}\in\GSp(\ZZ/N\ZZ)$, and
	\begin{equation}\label{eq:general}
	f(\tau)^{[\mathfrak{a}]} = 
	f^U (M\tau)
	= f^U (\tau').
	\end{equation}
\end{theorem}
For the computation of suitable $B$ and $C$
(and hence $\tau'$, $M$ and $U$), see
Section~\ref{sec:basis}.
We implemented the complete computation of $\tau'$, $M$ and $U$
in \cite{cmcode} as
$$\verb!tau.Shimura_reciprocity(a, N, period_matrix=True)!.$$
For the computation of $f^U$, see Section~\ref{sec:actioncompute}.
This makes $f^U(\tau')$ an explicit expression for
$f(\tau)^{[\mathfrak{a}]}$ that is suitable for computation.

\subsection{The class fields generated by complex multiplication}

Fix a primitive CM point $\tau$ and let the notation be as above.
The field
\[\cmext{N} = \reflexfield\big(f(\tau) : f\in\FcalN\ \mbox{s.t.}\ f(\tau)\not=\infty\big)\subset \CC.\]
is an abelian extension of~$\reflexfield$, and
we now describe 
the corresponding ideal group.

Let $F$ be the smallest positive integer
satisfying $F\mathcal{O}_K\subset\mathcal{O}$.
For~$x\in K$, we write
$x\equiv 1\modstar N\mathcal{O}$
to mean
$x=a/b$ where $a$ and $b\not=0$
are elements of~$\mathcal{O}$ that are invertible
modulo~$NF\mathcal{O}$
and congruent to each other modulo~$N\mathcal{O}$.
For various equivalent definitions,
see Definition~\ref{def:equivalence}.
This is equivalent to standard definitions in the case $\mathcal{O}=\mathcal{O}_K$.

\begin{theorem}\label{thm:idealgroup}
The extension $\cmext{N}/\reflexfield$
is abelian and of conductor dividing~$NF$.
Its Galois group is isomorphic via the Artin
isomorphism to the quotient group $I(NF) / \cmgp{N}$,
where $I(NF)$ is the group
of fractional $\mathcal{O}_{\reflexfield}$-ideals
with numerator and denominator coprime to $NF$,
and
\begin{equation}\label{eq:HPhiO}
\cmgp{N} = \left\{ \mathfrak{a}\in \IKr{NF} :
\exists \mu \in K\ \text{with}\ \begin{array}{c}
\NPhirO(\mathfrak{a}) = \mu\mathcal{O}\\
\mu\overline{\mu} = N(\mathfrak{a})\in\QQ\\
{\mu\equiv 1\modstar N\mathcal{O}}\end{array}\right\}.
\end{equation}
 \end{theorem}
\begin{remark}A similar result for fields
of definition of torsion points on normalized
Kummer varieties appears as
Main Theorem 3 in \S17 of~\cite{shimura-taniyama}
(see also Main Theorem 2 in \S16 of \cite{shimura-taniyama, shimura}).

A similar result in adèlic language for
fields of moduli of abelian varieties with torsion structure
appears as
Corollary 5.16 of~\cite{shimuraAF}
and Corollary 18.9 of~\cite{shimura}).

Our statement and proof are
directly in the language
of the fields $\FcalN$ using the reciprocity laws.
The proof is in Section~\ref{sec:proofidealgroup}.
\end{remark}
 Note that this theorem implies that $\cmext{N}$ depends only on $\mathcal{O}$ and~$\Phi$,
not on~$\tau$.

\begin{definition}\label{def:mu}
For $\mathfrak{a}\in \cmgp{1}$, we write
$\mu(\mathfrak{a})$ to denote any element
of $K^{\units}$ as in~\eqref{eq:HPhiO} with $N=1$.
\end{definition}
Note that $\mu(\mathfrak{a})$ is uniquely defined
up to multiplication by roots of unity in~$\mathcal{O}$.

\begin{algorithm}[Computing $\mu(\mathfrak{a})$ for the case $\mathcal{O}=\mathcal{O}_K$]\label{alg:mu}\hfill\\
	\textbf{Input:} $\reflextype$ and a fractional ideal $\mathfrak{a}$ of $\mathcal{O}_{\reflexfield}$.\\
	\textbf{Output:} The list of all
	elements $\mu\in K^{\units}$ such that $N_{\reflextype}(\mathfrak{a}) = \mu\mathcal{O}_K$
		and $\mu\overline{\mu} \in\QQ$.\\
	\textbf{Algorithm:}
	\begin{enumerate}
		\item Compute the class group and unit group of $K$. Compute the maximal totally real
		subfield $K_0$ of $K$ and its unit group $\mathcal{O}_{K_0}^{\units}$.
		Compute the quotient $\mathcal{O}_{K_0}^{\units}/N_{K/K_0}(\mathcal{O}_{K}^{\units})$.
		This can be done using e.g.~the algorithms of \cite{cohen},
		or the software Magma~\cite{magma} or PARI~\cite{pari}.
	PARI~\cite{pari} can be used through SageMath~\cite{sage}.
	\item Compute $N_{\reflextype}(\mathfrak{a})$ and test whether it is principal.
		\begin{enumerate}
	\item If it is, then let $\beta\in K^{\units}$ be a generator.
	\item Otherwise return an empty list.
\end{enumerate}
\item Let $u = \beta\overline{\beta} / N(\mathfrak{a}) \in \mathcal{O}_{K_0}^{\units}$
and test whether $u \in N_{K/K_0}(\mathcal{O}_{K}^{\units})$.
		\begin{enumerate}
	\item If it is, then take $v\in\mathcal{O}_K^{\units}$ such that $v\overline{v}=u$.
	\item Otherwise return an empty list.
\end{enumerate}
\item Return $\{w\beta/v \ :\ w\in(\mathcal{O}_{K}^{\units})^{\mathrm{tors}}\}$.
	\end{enumerate}
\noindent We implemented this as \verb!a_to_mus(Phir, a)! in~\cite{cmcode}.
\end{algorithm}
\begin{proof}[Proof of Algorithm~\ref{alg:mu}]
	It is clear that every $\mu = 
	w
	\beta /v$ in the output generates $N_{\reflextype}(\mathfrak{a})$
	and satisfies $\mu\overline{\mu}=N(\mathfrak{a})\in\QQ$.
	Conversely, suppose that $N_{\reflextype}(\mathfrak{a})=\mu\mathcal{O}_K$
	and $\mu\overline{\mu} \in\QQ$. Then $\mu\overline{\mu} = N(\mathfrak{a})$
	and $N_{\reflextype}(\mathfrak{a})$ is principal, so $\beta$ exists.

	Let $r =\beta/\mu\in\mathcal{O}_{K}^{\units}$.
	Then $r\overline{r} = \beta\overline{\beta} / N(\mathfrak{a}) = u$,
	hence $v$ exists.
	
	Let $w = v/r\in\mathcal{O}_{K}^{\units}$. Then $w\overline{w}=1$, so $w$
	is a root of unity. Therefore, $\mu = \beta/r = w\beta/v$ is listed by the algorithm.
\end{proof}

Using Algorithm~\ref{alg:mu} and standard algorithms for computing
ray class groups and computing quotients of groups, we can compute the group
$H_{\Phi,\mathcal{O}}(N)/\PKr{NF}$ as a
subset of the ray class
group $\mathrm{Cl}(NF) = \IKr{NF}/\PKr{NF}$
and in turn compute the group
$\mathrm{Gal}(\cmext{N}/\reflexfield) = \IKr{NF} / H_{\Phi,\mathcal{O}}(N)$.
For an efficient and detailed algorithm, see Asuncion~\cite{jaredpaper1,asuncionthesis}.

\subsection{Class invariants and a special case of the main theorem}\label{ssec:statementbeforelast}

The reciprocity law (Theorem~\ref{thm:general})
gives the Galois action of $\Gal(\cmext{N})/\reflexfield)$
on $f(\tau)$.
In order to decide whether $f(\tau)$ is in the field
$\cmext{1}$ generated by the values of Igusa invariants at~$\tau$,
we need only the Galois action of the subgroup $\Gal(\cmext{N})/\cmext{1})$.
For that particular subgroup, we have a simpler version of the reciprocity law as follows.

From Theorem~\ref{thm:idealgroup}, we have
$\Gal(\cmext{N}/\cmext{1})= (\IKr{NF}\cap \cmgp{1}) / \cmgp{N}$.
For any fractional ideal
$\mathfrak{a}\in \IKr{NF}\cap\cmgp{1}$,
we get an element $\mu=\mu(\mathfrak{a})\in K$ 
with $\mu\overline{\mu}=N(\mathfrak{a})\in\QQ$
and $\NPhirO(\mathfrak{a})=\mu\mathcal{O}$
(cf.~Definition~\ref{def:mu}).

\begin{theorem}\label{thm:special}
		Let $\tau=\tau(\Phi, \mathfrak{b}, \xi, B)\in\HH_g$ be a primitive CM point,
	let $N$ be a positive integer and let $f\in\mathcal{F}_N$
	be a function that does not have a pole at $\tau$.
	
	For any $\mathfrak{a}\in \IKr{NF}\cap\cmgp{1}$,
	we have
	$$f(\tau)^{[\mathfrak{a}]} = f^{\wasepsilon{\mu}}(\tau)$$
	where $\mu\in K$ is such that $\mu\overline{\mu} =N(\mathfrak{a})$
	and $\NPhirO(\mathfrak{a})=\mu\mathcal{O}$.
\end{theorem}

\newcommand{\bracketsO}{\wasepsilon{(\mathcal{O}^\units)^{\mathrm{tors}}}}
Observe that we have constructed a map
\label{sec:statementspecific}
\begin{align}\label{eq:wascalledg}
\wascalledg : \frac{\IKr{NF}\cap\cmgp{1}}{\cmgp{N}} \quad
\longrightarrow&  \quad \GSp(\ZZ/N\ZZ)/\bracketsO\\
[\mathfrak{a}]\quad \longmapsto & \quad\wasepsilon{\mu(\mathfrak{a})},\nonumber\end{align}
where $\bracketsO = \{\wasepsilon{u} : u\in (\mathcal{O}^\units)^{\mathrm{tors}}\}$.
The theorem then states
$$f(\tau)^{[\mathfrak{a}]} = f^{\wascalledg(\mathfrak{a})}(\tau).$$
\begin{remark}If $\mathfrak{a}$ is principal,
then the reciprocity map becomes even more explicit:
\begin{equation}\label{eq:moreexplicit}\wascalledg((\alpha)) = \wasepsilon{N_{\reflextype}(\alpha)}\quad\text{for all}\quad \alpha\in\reflexfield\ \text{with}\ (\alpha)\in \IKr{NF}.
\end{equation}
\end{remark}

We now get the following way to look for \emph{class invariants},
that is, values $f(\tau)$ with $f\in\Fcalall$ and $f(\tau) \in \cmext{1}$.
Given $\tau = \tau(\Phi, \mathfrak{b}, \xi, B)$, we compute the image $r(X)$
for a set of generators $X$ of the domain of~$r$.
Then $f(\tau)$ is a class invariant whenever $f$
is fixed by $r(X)$.

\begin{algorithm}[Computing the image of $r$]\label{alg:imageofr}\hfill\\
	\textbf{Input:} $N$, $F$, $\Phi$, $\mathfrak{b}$, $\xi$, $B$.\\
	\textbf{Output:} a complete set $R\subset \GSp(\ZZ/N\ZZ)$
	of representatives of the image $r(X)$ of
	a set of generators $X$ of the domain of $r$.\\
	\textbf{Algorithm:}
	\begin{enumerate}
		\item Compute $G = (I(NF)\cap H_{\Phi,\mathcal{O}}(1)) / \PKr{NF}\subset \mathrm{Cl}(NF)$.
		\item Let $X$ be a set of generators of $G$.
		\item For every element of $X$, choose a representative $\mathfrak{a}$,
		take an arbitrary $\mu$ in the output of Algorithm~\ref{alg:mu}
		and compute $\wasepsilon{\mu}$.
		Return the list of matrices $\wasepsilon{\mu}$ computed in this way.
\end{enumerate}
\end{algorithm}
We implemented this algorithm as \verb!reciprocity_map_image(tau, N)!
in~\cite{cmcode}.
We give an example in Section~\ref{ssec:detailedexample}.

\subsection{Complex conjugation}\label{ssec:conjugationstatement}

  Now assume that $f(\tau)$ is a class invariant, that is, is
in $\cmext{1}$.
  The coefficients of its minimal polynomial $H_f$ over~$\reflexfield$ are elements of~$\reflexfield$.
  If these coefficients are in the maximal totally real subfield
  $\realreflex\subset \reflexfield$, then they are easier
  to compute and take up even less space.
  We now give a sufficient criterion for these coefficients to be in~$\realreflex$.

  Let $\mathcal{M}:= \QQ(f(\tau) : f \in \Fcal{1}, f(\tau)\not=\infty)$
  be the \emph{field of moduli} of the
  principally polarized abelian variety corresponding to~$\tau$,
  and let $\mathcal{M}_0 = \mathcal{M}\realreflex\subset \mathcal{M}\reflexfield=\cmext{1}$.
  
  We give two results. Lemma~\ref{lem:complexconjugation} says that often
  $\mathcal{M}_0$ is strictly smaller than $\cmext{1}$.
  And if this is the case, then 
  Proposition~\ref{prop:complexconjugation} gives a criterion for
  the minimal polynomial of $f(\tau)$ over $\reflexfield$
  to have coefficients in $\realreflex$.

  \begin{lemma}\label{lem:complexconjugation}
     Suppose $\tau$ corresponds to a pair $(\mathfrak{b},\xi)$.
\begin{enumerate}
  \item  The degree of $\cmext{1}/\mathcal{M}_0$ equals~$2$
     if and only if there is an ideal $\mathfrak{a}\in I(F)$
     and an element $\mu\in K^\units$ such that
     $\NPhirO(\mathfrak{a})\overline{\mathfrak{b}} = \mu \mathfrak{b}$
     and $\mu\overline{\mu} \in\QQ$.
\item  If $g\leq 2$, $\mathfrak{b}$ is coprime to $F\mathcal{O}$,
and $\Phi$ is a primitive CM type, 
then the conditions in part~(1) are satisfied and we can take
\begin{enumerate}
\item $g=1$, $\mathfrak{a}=N_{\Phi}(\mathfrak{b}\overline{\mathfrak{b}}^{-1}\mathcal{O}_K)$
and $\mu=1$; or
  \item $g=2$, $\mathfrak{a}=N_{\Phi}(\mathfrak{b}\mathcal{O}_K)$
and~$\mu=N(\mathfrak{b})$.
\end{enumerate}
\item If $\mathfrak{b}=\mathcal{O}$, then the conditions
in part~(1) are satisfied and we can take $\mathfrak{a}=\mathcal{O}_{\reflexfield}$ and~$\mu=1$.
\end{enumerate}
  \end{lemma}
  
  \begin{remark}
  	The conditions in part (1) are equivalent to condition (5.1) of Enge-Streng~\cite{enge-streng}.
  	Proposition 5.3 of \cite{enge-streng} gives some cases in which they hold for $g=3$ 
  	and $g=6$.
  	\end{remark}
  \begin{proposition}\label{prop:complexconjugation}
  	Given $\tau = \tau(\Phi, B)$ and $f\in \FcalN$,
  	assume
  $\deg \cmext{1}/\mathcal{M}_0=2$
  and $f(\tau)\in\cmext{1}$.

  Let $(\mathfrak{a}, \mu)$ be as in Lemma~\ref{lem:complexconjugation}(1)
  and assume
  that $\mathfrak{a}$ is coprime to~$N$.
  Write $B=(b_1,\ldots,\allowbreak b_g,\allowbreak b_{g+1},\allowbreak \ldots,b_{2g})$
  and consider the $\QQ$-basis
    $Q=(\mu^{-1} \overline{b_1},\ldots,\allowbreak 
  \mu^{-1}\overline{b_g},\allowbreak 
  \mu^{-1}\overline{b_{g+1}},\allowbreak \ldots,
  \mu^{-1}\overline{b_{2g}})$ of~$K$.
  
  Then $\leftchange{Q}{B}$ is invertible modulo~$N$ with inverse
  $V\in \GSp(\ZZ/N\ZZ)$.
  Moreover, the following are equivalent:
  \begin{enumerate}
  \item $f(\tau)\in \mathcal{M}_0$,
  \item $f^{V}({\tau})=f({\tau})$.
  \end{enumerate}
  If these conditions are satisfied, then
  the minimal polynomial of $f(\tau)$ over
  $\reflexfield$ has coefficients in~$\realreflex$.
  \end{proposition}

The assumption that $\mathfrak{a}$ be coprime to~$N$ is without loss of generality.

\begin{example}
	Suppose $g=1$ and $\mathfrak{b} = \ZZ[\sqrt{D}]$.
	Then we can take $b_1=\sqrt{D}$ and $b_2=1$, and by Lemma~\ref{lem:complexconjugation}
	also $\mu=1$, so $c_1=-b_1$ and $c_2=b_2$, hence
	$M$ is the diagonal $2\times 2$ matrix $\begin{psmallmatrix}-1&0\\0&1\end{psmallmatrix}$ and so is~$V$.
	As the matrix $-I\in\SL_2(\ZZ)$ acts trivially on every $\tau\in\HH_1$,
	we find that $V$ acts exactly as $i(-1\ \mathrm{mod}\ N)$ does, which
	is as complex conjugation of the coefficients of~$f$.
	The condition \ref{prop:complexconjugation}(2) then translates
	to $f$ having only real coefficients in its $q$-expansion.
\end{example}

  We will prove Lemma~\ref{lem:complexconjugation} and
  Proposition~\ref{prop:complexconjugation} in Section~\ref{sec:conjugationproof}.
  These results show that, if we restrict to~$f$ that satisfy~$f^{V} = f$,
  then the minimal polynomial of $f(\tau)$
  over~$\reflexfield$ is defined over~$\realreflex$.
  We implemented the computation of the matrix $V$ in~\cite{cmcode} as
  \verb!tau.complex_conjugation_symplectic_matrix(N)!.
 
 
 \section{The ad\`elic version} \label{sec:ad}
 
 In Section~\ref{sec:proof}, we give a proof of the results
 stated in Sections \ref{ssec:statementfirst}--\ref{ssec:statementbeforelast}
 (including the reciprocity law).
 For this, we use Shimura's own formulation of his reciprocity law,
 which we state in Section~\ref{sec:ad}.
 In Section~\ref{sec:conjugationproof}, we prove the results
 about complex conjugation stated in Section~\ref{ssec:conjugationstatement}.
 
 The reader who is not interested in the proof, or would
 like to see the applications first, is advised
 to skip ahead and read Sections~\ref{sec:theta} (Theta constants),
 \ref{sec:examples} (Examples) and \ref{sec:applications}
 (Applications),
 first. They are independent of Sections~\ref{sec:ad}--\ref{sec:conjugationproof}.
 
Shimura developed
his reciprocity laws for various types of multivariate modular functions, modular forms,
and theta functions
in a series of articles \cite{shimura-models-I, shimura-models-II, shimura-arithmetic, shimura-fourier, shimura-theta-cm, shimura-reciprocity-theta}.
See also the textbook \cite[26.10]{shimura}.
Rather than reproving the reciprocity law in our setting, we will
quote a streamlined version stated by Shimura in
the language of id\`eles
and rework it (in Section~\ref{sec:proof})
into a version with ideals
and a more explicit group action.
This means that our proof
will not be the most direct proof, as
the ad{\`e}lic statement mashes all levels~$N$ together, and
we take them apart again; and Shimura's original series of
articles
starts with theta functions, while we give
them as a special case afterwards
(Section~\ref{sec:theta}).
However, our approach does allow us to give both the
computationally practical statement and the elegant ad{\`e}lic statement,
explain how they are related, and keep the proofs
short at the same time.

The reader who would rather see a direct proof of our explicit
version of the reciprocity
law should see  Yang~\cite{yang-shimura}.
Yang, inspired by our explicit statements, gives a direct
proof of our explicit version of Shimura's reciprocity law
(Theorem~4.1 of~\cite{yang-shimura} is our Theorem~\ref{thm:general})
and uses that to prove the ad\`elic statement.

We start by citing Shimura's ad\`elic action of~$\GSp$,
and linking it to the actions of Proposition~\ref{prop:groupaction}.

 Let $\AA$ be the ring of ad\`eles of~$\QQ$ and call an element of its unit group
 \emph{positive} if its $\RR$-component is positive.
Let $\ZZhat=\lim_{\leftarrow} \ZZ/N\ZZ$ be the ring
of finite integral ad\'eles, so $\AA = (\ZZhat\otimes\QQ) \times \RR$.
Let $\Fcalall = \cup_{N}\FcalN$.
\begin{proposition}\label{prop:biggroupaction}
There is a unique right action of 
$\GSp(\AA)^+$ on $\Fcalall$
satisfying
\begin{enumerate}
\item for~$x\in\AA^\units$ and $f\in\Fcalall$, we define $f^{i(x)}$ as
the function obtained from $f$ by acting with $x^{-1}$ on the $q$-expansion coefficients,
 \item for~$A\in \GSp(\QQ)^+$, $f\in\Fcalall$, $\tau\in\HH_g$, we have $f^A(\tau) = f(A \tau)$,
 \item for any~$N$, the group
$\wascalledS{} = \{A\in \GSp(\ZZhat) : A \equiv 1\modstar N\}\times \GSp(\RR)^+$
acts trivially on the subfield~$\FcalN$,
where we write $A\equiv 1\modstar N$ if and only if
for all $p\mid N$ we have $A_p\in  1 + N\matrixring{2g}{\ZZ_p}$.
 \end{enumerate}
 \end{proposition}
\begin{proof}
Existence is a special case of~\cite[Thm.~5(v,vi,vii)]{shimura-fourier}.
Uniqueness follows from the proof of~\cite[Proposition~1.3]{shimura-reciprocity-theta}.
\end{proof}
\begin{remark}\label{rem:aboutproof}
Our reference for existence in Proposition~\ref{prop:biggroupaction},
though directly applicable to our situation,
may not be satisfactory to some readers,
as the paper
does not contain the full proof.
Therefore, just like~\cite{shimura-fourier}, we give some pointers for the proof.
The action is constructed in~\cite[Section 2.7]{shimura-models-I}
for a field $k_S(V_S)$.
The field $k_S(V_S)$ is defined without $q$-expansions,
hence that reference only contains a weak version of~(1),
but (2) is \cite[(2.7.2)]{shimura-models-I} and (3) follows immediately
from \cite[{(2.5.3$_\mathrm{a}$)}]{shimura-models-I}.
Our stronger version of~(1), as well as the link between 
$\Fcalall$ and~$k_S(V_S)$,
is given in~\cite{shimura-fourier}.
Both that reference and \cite[\S6]{shimura-arithmetic}
claim that the proof is exactly the same as in the Hilbert modular case,
which is~\cite{shimura-arithmetic}.
\end{remark}

The following corollary proves exactly Proposition~\ref{prop:groupaction}.

\begin{corollary}\label{cor:zzmod}
The action of Proposition~\ref{prop:biggroupaction}
has the following property:
\begin{enumerate}
\setcounter{enumi}{3}
\item For any positive integer $N$, any $f\in\FcalN$, and any $$A=(A_{\mathrm{f}}, A_{\infty})\in\GSp(\ZZhat)\times \GSp(\RR)^+ \subset \GSp(\AA)^+,$$ we have $f^A\in\FcalN$,
and $f^A$ depends only on $(A_{\mathrm{f}}\mod N)\in\GSp(\ZZ/N\ZZ)$. Moreover, the induced action of~$\GSp(\ZZ/N\ZZ)$ on $\FcalN$
is exactly as in Proposition~\ref{prop:groupaction}.
\end{enumerate}
\end{corollary}
\begin{proof}
The inclusion $f^A\in\FcalN$ follows from the construction
of the action (see \cite[(2.5.3)]{shimura-models-I} and~\cite{shimura-fourier}).
That~$f^A$ depends only on $(A_{\mathrm{f}} \mod N)$ is Proposition~\ref{prop:biggroupaction}(3).
It follows that the action induces an action of~$\GSp(\ZZ/N\ZZ)$ on~$\FcalN$.
To prove that this action is as in Proposition~\ref{prop:groupaction},
it remains only to compute this action for~$B\in \Sp(\ZZ/N\ZZ)$ and for~$B=i(t)$ with $t\in(\ZZ/N\ZZ)^\units$.

In the case~$B\in\Sp(\ZZ/N\ZZ)$, we lift $B$ to $A'\in \Sp(\ZZ)$
(possible by \cite[Theorem~VII.21]{newman}),
and we get $f^B = \smash{f^{A'}}$ by Proposition~\ref{prop:biggroupaction}(3).
As we have $$\Sp(\ZZ) = \GSp(\QQ)^+ \cap (\GSp(\ZZhat)\times \GSp(\RR)^+),$$
we can then apply Proposition~\ref{prop:biggroupaction}(2)
to get that $f^B$ is as in Proposition~\ref{prop:groupaction}.

In the case~$B=i(t)$ with $t\in(\ZZ/N\ZZ)^\units$, we lift~$t$ to $x\in \ZZhat^\units$ and apply 
Proposition~\ref{prop:biggroupaction}(1), which gives
$f^B = f^{i(x)} = f^{x^{-1}} = f^t$,
so that again $f^B$ is as in Proposition~\ref{prop:groupaction}.
Here, the switch from $x^{-1}$ to $t$ is explained by the usual map
from the id\`ele class group to the ray class group:
starting from $x\in \ZZhat^\units$, we take $c\in \QQ^\units\cap \ZZ$
with $t = (c\bmod N)$
to get $cx\in \QQhat^\units$ that is $1$ modulo $N$ and in the same id\`ele class.
Then $cx$ in turn maps to the class of the fractional ideal $(c)$ in the ray class group,
which acts as~$t$ on $\QQ(\zeta_N)$.
\end{proof}

 Let $\tau = \tau(\Phi, B)\in \HH_g$ be a primitive CM point
 for the CM field~$K$.

The type norm $N_{\reflextype}$ and the map
$\epsilon : a\mapsto \leftmult{a}{B}{B}$ induce
ad\`elic maps $N_{\reflextype}:\reflexid\rightarrow \Kid$
and $\epsilon: \Kid\rightarrow \GL_{2g}(\AA)$
and the composite map sends $\reflexid$ to $\GSp(\AA)^+$.
Shimura gives the following reciprocity law, stated
in a very sleek manner using
the action of Proposition~\ref{prop:biggroupaction}.
\begin{theorem}[Shimura]\label{thm:shimrecipad}
Let $\tau$ and the notation be as above.
Then for every $f\in\Fcalall$ such that~$f(\tau)$ is finite and every $x\in \reflexid$, we have
$$f(\tau) \in \reflexab\quad\mbox{and}\quad f(\tau)^x = f^{\matrixofShimura{x}^{-1}}(\tau).$$
\end{theorem}
\begin{proof}
This is equation (3.43) of  \cite[p.~57]{shimura-reciprocity-theta}
up to two minor modifications.

First of all, that reference assumes that
the abelian variety $A=\CC^g/(\tau\ZZ^n+\delta \ZZ^n)$
for an integer $\delta\geq 3$ has CM,
but that variety has CM by $K$ if and only if ours has.

Secondly, 
the matrix $\epsilon(a)$ is defined differently
in~\cite{shimura-reciprocity-theta}, namely for $a\in K$
by the (less computationally convenient) identity of complex matrices
\begin{equation}\label{eq:epsilonShim}\rho(a)(\tau, 1_{g\times g}) = (\tau, 1_{g\times g})\epsilon(a)\transpose
\end{equation}
where $\rho(a)\in\matrixring{g}{\CC}$ is the matrix of $a\in K = \End(A)\otimes \QQ$
with respect to the standard basis of~$\CC^g$.
We now check that our
matrix $\epsilon(a) = \leftmult{a}{B}{B}$ also satisfies~\eqref{eq:epsilonShim}.
We have
$a B = \leftmult{a}{B}{B} B$, which by taking the transpose and
applying $\Phi$ leads to
$$\mathrm{diag}(\Phi(a)) (\Phi(b_1),\ldots,\Phi(b_{2g})) = (\Phi(b_1),\ldots,\Phi(b_{2g}))(\leftmult{a}{B}{B})\transpose.$$
The change of basis $(\Phi(b_{g+1}),\ldots, \Phi(b_{2g}))^{-1}$
yields~\eqref{eq:epsilonShim} for $\epsilon(a) = \leftmult{a}{B}{B}$.
\end{proof}

 
\begin{remark}
Recent work of Hertogh~\cite{hertogh, hertoghcode} provides a computer implementation
of ad\`eles and id\`eles.
It would be interesting to see whether this allows one to
use Theorem~\ref{thm:shimrecipad} directly
in a practical way on a computer, and whether this
can be made to work as well in practice as our main theorems.
\end{remark}

   \section{Proof of the explicit reciprocity law} \label{sec:proof}

   In this section, we prove our explicit version
   of Shimura's reciprocity law,
   using Shimura's ad{\`e}lic version (Theorem~\ref{thm:shimrecipad}).
   
The bridge between ad\`elic and ideal theoretic class field theory
is the surjection
\begin{equation}\label{eq:idtoid}
	\reflexid/ \reflexunits 
	\rightarrow \mathrm{Cl}(NF) = \IKr{NF} / \PKr{NF}\end{equation}
that maps the class of an id\`ele $x\equiv 1\modstar NF$ to
the class of the ideal $\mathfrak{a}$ with $\ord_{\mathfrak{p}}(\mathfrak{a}) = \ord_{\mathfrak{p}}(x_{\mathfrak{p}})$.

   Given $f\in\FcalN$ and an id\`ele $x\in \reflexid$, let $[\mathfrak{a}]$
   be the image of $x$ under the map~\eqref{eq:idtoid}.
   By Theorem~\ref{thm:shimrecipad}, we have
   $f(\tau)^{[\mathfrak{a}]} = f(\tau)^x=f^{\matrixofShimura{x}^{-1}}(\tau)$,
   and our goal is to express this in terms of~$\mathfrak{a}$.
  To do so, we
  write $\matrixofShimura{x}^{-1}
   = S\ U\ M$ with $M\in\GSp(\QQ)^+$, $U\in\GSp(\ZZhat)$, $S\in\Stab_f$,
  and both $M$ and $(U\ \text{mod}\ N)$ explicit in terms of~$\mathfrak{a}$.
  Then we can conclude $f^{\mathfrak{a}}(\tau) =
   f^{(U\ \text{mod}\ N)}(M\tau)$,
   by
  Theorem~\ref{thm:shimrecipad}.
  \begin{remark}
  The strong approximation theorem for~$\GSp(\AA)$
  in fact tells us 
  that such a decomposition always exists, even with $U\in i(\smash{\ZZhat^\units})$
  (\cite[Lemma~1.1]{shimura-reciprocity-theta}).
  However, as in the genus-one case~\cite{gee-stevenhagen},
  we will be satisfied with having only
  $U\in \GSp(\smash{\ZZhat})$.
  In fact, by allowing $U\in\GSp(\smash{\ZZhat})$, we can make sure that $M\tau$
  is in a fundamental domain for $\Sp(\ZZ)$, which improves the speed
of convergence in practical computations.
  \end{remark}
  

\subsection{Coprimality and congruence for fractions}

To help in translating ad{\`e}lic statements to more
concrete statements, we first state
some equivalent definitions
of ``$\!\modstar{}\!\!$'' that we will use.
This is not new, but statements that apply to non-maximal
orders are rare in the literature, so we give a detailed statement and proof.

\newcommand{\Op}{\mathcal{O}_{(p)}}
Let $\mathcal{O}$ be an order in $K$ and let $F\in\ZZ$ be
the smallest
positive integer such that $\mathcal{O}\supset F\mathcal{O}_K$.
For any prime number $p\in\ZZ$, let
\[\Op = \{a/b\in K : a\in\mathcal{O}, b\in\ZZ\setminus p\ZZ\}.\]
In this section, for $a\in\mathcal{O}$, we
use the notation
$\overline{a}=(a\ \mathrm{mod}\ NF\mathcal{O})\in (\mathcal{O}/NF\mathcal{O})$.

\begin{definition}\label{def:equivalence}
	Let $N$ be a positive integer. We say that an element
	$x \in K^{\units}$ is \emph{coprime to $NF$ with respect to $\mathcal{O}$} if one of the following equivalent
	conditions holds (equivalence is proven below):
	\begin{enumerate}
		\item $x=a/b$ for some $a,b\in\mathcal{O}$ 
		with $\overline{a}, \overline{b} \in(\mathcal{O}/NF\mathcal{O})^{\units}$
		and $b\not=0$,
		\item $x = a/b$ for some $a\in\mathcal{O}$
		and $b\in\ZZ\setminus\{0\}$ with
		$\overline{a}\in(\mathcal{O}/NF\mathcal{O})^{\units}$
		and
		$b-1\in NF\ZZ$,
		\item for all prime numbers $p\mid NF$, we have
$x \in \Op^{\units}$,
		\item
		$x\mathcal{O} = \mathfrak{a}\mathfrak{b}^{-1}$
		for non-zero $\mathcal{O}$-ideals $\mathfrak{a}$ and $\mathfrak{b}$
		that are coprime to $NF$ in the sense that
		$\mathfrak{a}+NF\mathcal{O}=\mathfrak{b}+NF\mathcal{O}=\mathcal{O}$.
	\end{enumerate}
	We write $x \equiv 1\modstar N\mathcal{O}$ to mean
	that one of the following equivalent conditions holds
   (equivalence is proven below):
	\begin{itemize}
		\item[(1')] as in (1) above, with additionally
		$a-b\in N\mathcal{O}$,
		\item[(2')] as in (2) above, with additionally
$a-1\in N\mathcal{O}$,
		\item[(3')] as in (3) above, with additionally
		$x -1 \in N\Op$ for all prime numbers $p\mid N$.
	\end{itemize}
	In terms of $p$-adic numbers, both 
	$\mathcal{O}\otimes \ZZ_p$
	and $K$ are subrings of $\mathcal{O}\otimes \QQ_p$,
	and their intersection is exactly $\Op$.
	In particular, the conditions (3) and (3') can equivalently
	be written with $\mathcal{O}\otimes\ZZ_p$ instead of $\Op$.
\end{definition}

\begin{proof}[Proof of equivalence in Definition~\ref{def:equivalence}]
We start with the equivalence of (1)--(4).\\
$(2)\Rightarrow (1)$ is obvious.\\
$(1)\Rightarrow (2)$.
Let $N(b) = \#(\mathcal{O}/b\mathcal{O})$.
We start by showing that $NF$ is coprime to $N(b)$ and that $N(b)$ is an $\mathcal{O}$-multiple of~$b$.

We have $1\in b\mathcal{O}+NF\mathcal{O}$,
so $NF$ is a unit modulo $b\mathcal{O}$,
hence multiplication by $NF$ is invertible on the additive group $(\mathcal{O}/b\mathcal{O})$
of order $N(b)$, so $NF$ is coprime to~$N(b)$.
Note that $N(b)$ annihilates the group $\mathcal{O}/b\mathcal{O}$,
hence $N(b)\in b\mathcal{O}$, so $N(b)$ is a multiple of~$b$.

Let $c\in\ZZ$ be such that
$b':=cN(b)\equiv 1\pmod{NF}$.
We get $x=a'/b'$ with $a'=acN(b)/b\in\mathcal{O}$ and $\overline{a'}\in(\mathcal{O}/NF\mathcal{O})^\units$.
\\
$(1)\Rightarrow (3)$. Suppose that $x$ satisfies~$(1)$.
Then $x^{-1}$ also satisfies~$(1)$.
By ``$(1)\Rightarrow (2)$'', we then get that both $x$ and $x^{-1}$ satisfy~$(2)$.
By the definition of $\mathcal{O}_{(p)}$, we then get
$x, x^{-1}\in\mathcal{O}_{(p)}$,
hence $x\in\mathcal{O}_{(p)}^{\units}$.\\
$(3)\Rightarrow (4)$. If $NF=1$, then this is trivial, so suppose~$NF>1$.

For every prime $p\mid NF$, write $x = a_p/b_p$ and $x^{-1} = c_p/d_p$ with $a_p, c_p\in\mathcal{O}$
and $b_p,d_p\in\ZZ\setminus p\ZZ$.
Let $\mathfrak{b} = \sum_{p} b_p\mathcal{O}\subset\mathcal{O}$,
 $\mathfrak{d} = \sum_{p} d_p\mathcal{O}\subset\mathcal{O}$,
$\mathfrak{a} = \mathfrak{b}x = \sum_{p} a_p\mathcal{O}\subset \mathcal{O}$,
and $\mathfrak{c} = \mathfrak{d}x^{-1} = \sum_{p} c_p\mathcal{O}\subset \mathcal{O}$.
We get $\mathfrak{a}\mathfrak{c} = \mathfrak{b}\mathfrak{d}$.
The ideals $\mathfrak{b}$ and $\mathfrak{d}$ are coprime to all 
prime numbers $p\mid NF$ because of $b_p\in \mathfrak{b}$
and $d_p\in \mathfrak{d}$.
It follows that $\mathfrak{a}\mathfrak{c} = \mathfrak{b}\mathfrak{d}$ is coprime to $NF\mathcal{O}$.
In particular, both $\mathfrak{a}$ and $\mathfrak{b}$ are coprime to $NF\mathcal{O}$,
hence $\mathfrak{b}$ is invertible and we have $x\mathcal{O} = \mathfrak{a}\mathfrak{b}^{-1}$.\\
$(4)\Rightarrow (1)$. 
Suppose $x = \mathfrak{a}\mathfrak{b}^{-1}$
	with $\mathfrak{a}$ and $\mathfrak{b}$ non-zero ideals of
	$\mathcal{O}$ coprime to $NF\mathcal{O}$.
	Then $\mathfrak{a}$ and $\mathfrak{b}$ are both invertible
	and we have $x\mathfrak{b}=\mathfrak{a}$.
	We have $\mathfrak{b} + NF\mathcal{O} = \mathcal{O}$, hence there exists
	a $b\in\mathfrak{b}$ with $b\equiv 1\ (\mathrm{mod}\ NF\mathcal{O})$.
	Take a non-zero such~$b$.
	Let $a = bx$. Then $a\mathcal{O}=(b\mathfrak{b}^{-1})\mathfrak{a}$
	is coprime to $NF\mathcal{O}$.
	This proves~(1).
	
We have now proved that (1)--(4) are equivalent.
It remains to prove that (1')--(3') are equivalent.
Note that each ($n$') implies ($n$),
so we may and will assume that (1)--(4) hold.
Write $x=a/b$ as in (1) and $x=a'/b'$ as in~(2).
We have $ab'=a'b$, hence $b'(a-b) = b(a'-b')$.
As $b$ and $b'$ are invertible in $\mathcal{O}/NF\mathcal{O}$,
we get
$(1')\Leftrightarrow b'(a-b)\in N\mathcal{O}
\Leftrightarrow b(a'-b')\in N\mathcal{O}\Leftrightarrow (2')$.

Next, we have $b'(x-1) = (a'-1) - (b'-1)$
and for all $p\mid N$ we have
$b'\in\smash{\mathcal{O}_{(p)}^{\units}}$
and $b'-1\in NF\ZZ\subset N\mathcal{O}_{(p)}$.
In particular, we have (3')
if and only if for all $p\mid N$ we have $a'-1\in N\mathcal{O}_{(p)}$.
We also have $a'-1\in\mathcal{O}$ and
$\cap_{p\mid N} N\mathcal{O}_{(p)}\cap \mathcal{O} = N\mathcal{O}$,
so (3')$\Leftrightarrow$(2').
\end{proof}


  \subsection{The conductor}\label{sec:conductor}

We now prove the first statement in Theorem~\ref{thm:general}:
that $f(\tau)$ lies in the ray class field for the modulus~$NF$.
In other words, we prove that the extension
$\cmext{N}=\reflexfield(f(\tau) : f\in\FcalN)$ of~$\reflexfield$,
which is abelian by Theorem~\ref{thm:shimrecipad},
has conductor dividing~$NF$.

As in Section~\ref{sec:statement},
let $\mathfrak{b}$ be a fractional $\mathcal{O}$-ideal
with $\End(\mathfrak{b})=\mathcal{O}$
and let $F$ be the smallest positive integer such that
$F\mathcal{O}_K\subset\mathcal{O}$.
Let $B$ be a $\ZZ$-basis of $\mathfrak{b}$.
\begin{lemma}\label{lem:epsiloninteger}
For~$a\in K^\units$, we have $a\in\mathcal{O}$ if and only if $\leftmult{a}{B}{B}\in\matrixring{2g}{\ZZ}$.
\end{lemma}
\begin{proof}
We have $a\in\mathcal{O}$ if and only if $a\mathfrak{b}\subset\mathfrak{b}$,
which is equivalent to $\leftmult{a}{B}{B}\in\matrixring{2g}{\ZZ}$.
\end{proof}

\begin{lemma}\label{lem:epsilonmodstar}
For~$a\in K$, we have $a\equiv 1\modstar N\mathcal{O}$ if and only if
the following two conditions hold:
\begin{enumerate}
	\item we have $\leftmult{a}{B}{B}\in \GL_{2g}(\ZZ_p)$
	for all $p\mid NF$, and
	\item the coefficient-wise reduction modulo $N$ of 
$\leftmult{a}{B}{B}$
is the identity matrix.
\end{enumerate}
\end{lemma}
\begin{proof}
Lemma~\ref{lem:epsiloninteger} and its proof stay
valid when considered locally at a prime number~$p$,
that is,
replacing $\mathcal{O}$ by
 $\mathcal{O}_{(p)}$
and $\ZZ$ by $\ZZ_{(p)}=\QQ\cap \ZZ_p$ for a prime~$p$.
By Definition~\ref{def:equivalence}(3'), 
we have $a\equiv 1\modstar N\mathcal{O}$ if and only if
$(a-1)/N\in 
\mathcal{O}_{(p)}$ for all~$p\mid N$
and $a\in\smash{\mathcal{O}_{(p)}^\units}$ for all $p|NF$.
The result follows if we apply Lemma~\ref{lem:epsiloninteger}
to $(a-1)/N$ locally at all primes dividing~$N$ and to $a$
and $a^{-1}$ at all primes dividing~$NF$.
\end{proof}

\begin{proposition}\label{prop:conductor}
The conductor of $\cmext{N}$ divides~$NF$.
\end{proposition}
\begin{proof}
What we need to prove is equivalent to the statement that
the kernel
$$W_{NF}=\{x K^\units \in \reflexid/K^\units :
  x\equiv 1 \modstar NF\mathcal{O}_{\reflexfield}, \forall_{\mathfrak{p}}\ord_{\mathfrak{p}}(x)=0\}$$
of \eqref{eq:idtoid}
acts trivially on all $f\in \FcalN$.
So take $x$ with these properties and let
$y=(\leftmult{N_{\reflextype}(x)}{B}{B})^{-1}$.
Then Theorem~\ref{thm:shimrecipad} tells us 
$f(\tau)^x = f^y(\tau)$.

We have $N_{\reflextype}(x)\equiv 1\modstar NF\mathcal{O}_K$,
hence $N_{\reflextype}(x)\equiv 1\modstar N\mathcal{O}$.
Then by 
Lemma~\ref{lem:epsilonmodstar}
and another local application of Lemma~\ref{lem:epsiloninteger}
(this time to $p\nmid N$), we find
that~$y$ is in the group $\wascalledS{} = \{A\in \GSp(\ZZhat) : A \equiv 1\modstar N\}\times \GSp(\RR)^+$,
which acts trivially on $f$ by Proposition~\ref{prop:biggroupaction}.
So we get $f^y=f$, hence $f(\tau)^x = f^y(\tau) = f(\tau)$.
\end{proof}

\subsection{Changes of symplectic bases}\label{sec:mundane}

The next statement in Theorem~\ref{thm:general} is that the matrix~$M=\leftchange{C}{B}$
is in the group~$\GSp(\QQ)^+$.
Lemma~\ref{lem:mundanenew} proves this claim.

\begin{lemma}\label{lem:premundane}
For a field $F$ and matrix $M \in \GL_{2g}(F)$,
the following are equivalent:
	\begin{enumerate}[(1)]
	\item there exists $y\in F^\units$ such that $y M \Omega M\transpose = \Omega$,
	\item $M\in\GSp(F)$.
\end{enumerate}
If this is the case, then $\nu(M) = y^{-1}$.
\end{lemma}
\begin{proof}
	Statement (2) means $M\transpose \Omega M = \nu(M) \Omega$
	for some $\nu(M)\in F^\units$.
	By taking inverses and observing $\Omega^{-1} = -\Omega$,
	we see that this is equivalent to
	$M^{-1}\Omega (M\transpose)^{-1} = \nu(M)^{-1}\Omega$.
	Multiplying on the left by $M$ and on the right by $M\transpose$
	shows that this is equivalent to~(1), with $y = \nu(M)^{-1}$.
\end{proof}
\begin{lemma}\label{lem:mundanenew}
	Given $\tau = \tau(\Phi, \mathfrak{b}, \xi, B)$ and $M\in \GL_{2g}(\QQ)$,
	let $\mathfrak{c}$ be the subgroup of $K$ generated by $C = MB$ and let $E = E_{\xi}$.
	Then the following are equivalent:
	\begin{enumerate}[(1)]
		\item there exists $y\in\QQ^\units$ such that $yE$ is a principal polarization for $\mathfrak{c}$
		and $C$ is a symplectic basis of $\mathfrak{c}$ for~$yE$,
		\item $M\in\GSp(\QQ)^{+}$.
	\end{enumerate}
    Moreover, if this is the case, then we have
    \begin{enumerate}[(a)]
    	\item $\nu(M) = y^{-1}$, and
    	\item $\tau(\Phi, \mathfrak{c}, y\xi, C) = M\tau$.
    \end{enumerate}
\end{lemma}
\begin{proof}
	Since $C$ is a basis of~$\mathfrak{c}$, statement (1)
	is statement (1) of Lemma~\ref{lem:premundane} together with
	positive-definiteness of 
	$(u,v)\mapsto yE(iu,v)$.
	This positive-definiteness is equivalent to $y>0$, hence
	Lemma~\ref{lem:premundane} gives equivalence of (1) and~(2), as well as~(a).

Let $\tau' = \tau(\Phi, \mathfrak{c}, y\xi, C)$.
It remains to show $\tau' = M\tau$.
Write $M = (a,b; c, d)$ for $g\times g$ blocks $a,b,c,d$.
Write
$B = (b_1,\ldots, b_{2g})$, and
take the $g\times 2g$ matrix $\mathcal{B}=(B_1|B_2) = (\Phi(b_1)\mid\ldots\mid \Phi(b_{2g}))$,
and similarly define $\mathcal{C}$ using~$C$.
We have $C = MB$, hence $\mathcal{C} = \mathcal{B} M\transpose
= (B_1a\transpose+B_2b\transpose|B_1c\transpose+B_2d\transpose)$.
This gives
$\tau' = (B_1c\transpose+B_2d\transpose)^{-1} (B_1a\transpose+B_2b\transpose)$.
As $\tau$ and $\tau'$ are symmetric, we get $\tau = B_1\transpose (B_2\transpose)^{-1}$
and $\tau' = (aB_1\transpose+bB_2\transpose)(cB_1\transpose+dB_2\transpose)^{-1} = (a\tau + b) (c\tau + d)^{-1} = M\tau$.
\end{proof}

\subsection{Decomposing \texorpdfstring{$\matrixofShimura{x}$}{\matrixofShimurapdfstring{x}} modulo the stabilizer}\label{ssec:decomp}

Let us recall the situation of the theorem we
are proving (Theorem~\ref{thm:general}):
we have a fractional $\OKr$-ideal $\mathfrak{a}$ coprime to~$NF$,
a symplectic basis $B = (b_1,\ldots,b_g)$
of $\mathfrak{b}$ with respect to $E_{\xi}$,
and a symplectic basis $C = (c_1,\ldots,c_g) = BM\transpose$
of $\NPhirO(\mathfrak{a})^{-1}\mathfrak{b}$
with respect to $E_{N(\mathfrak{a})\xi}$.
Here $M = \leftchange{C}{B}$.

In order to compute the action of the ray class $[\mathfrak{a}]$ of $\mathfrak{a}$
modulo $NF$, we choose an id\`ele $x$ whose class maps to~$[\mathfrak{a}]$.
To be precise, we choose $x\in \reflexid$
such that
\begin{enumerate}
\item for every prime ideal $\mathfrak{p}\mid NF$ we have
$x_{\mathfrak{p}}=1$,
\item for all other prime ideals $\mathfrak{p}$ we have
$\ord_{\mathfrak{p}}(x_{\mathfrak{p}})=\ord_{\mathfrak{p}}(\mathfrak{a})$.
\end{enumerate}

Then by Lemma~\ref{lem:epsilonmodstar} we have
\begin{equation}\label{eq:trivialbutimportant}
\matrixofShimura{x} \equiv 1_{2g} \modstar NF,
\end{equation}
with $\epsilon : a\mapsto \leftmult{a}{B}{B}$ as defined above Theorem~\ref{thm:shimrecipad}.

\begin{lemma}\label{lem:defofa}
The matrix $A:=\matrixofShimura{x}^{-1}M^{-1}$ lies in $\GSp(\ZZhat)\times \GSp(\RR)^+$.
\end{lemma}
\begin{proof}
Note $\nu\circ \epsilon\circ N_{\reflextype} = N_{\reflexfield/\QQ}$, and the fact that~$\reflexfield$
has no real embeddings implies $N_{\reflexfield/\QQ}(\reflexfield\otimes \RR)\subset \RR_{\geq 0}$,
so $\matrixofShimura{x}_{\infty}\in\GSp(\RR)^+$.
We also have $M\in\GSp(\QQ)^+$ by Lemma~\ref{lem:mundanenew}, hence $A_\infty \in \GSp(\RR)^+$.
It now suffices
to prove for every prime number $p$ that~$A_p$ is in $\GSp(\ZZ_p)$.
For any number field~$L$ and~$x\in L_{\AA}^\units$,
write $x_p\in L \otimes\ZZ_p$ for the part corresponding
to primes over~$p$.

We have the following identity of~$\ZZ_p$-submodules of~$K\otimes\ZZ_p$ of rank $2g$:
$$(N_{\reflextype}(\mathfrak{a})^{-1}\mathfrak{b})\otimes \ZZ_p=N_{\reflextype}(x)_p^{-1} (\mathfrak{b}\otimes\ZZ_p)$$
(indeed, for $p\mid FN$, both sides are equal to
$\mathfrak{b}\otimes\ZZ_p$, while for $p\nmid F$,
the order is locally maximal and the identity follows from
$\ord_{\mathfrak{p}}(\mathfrak{a})=\ord_{\mathfrak{p}}(x_{\mathfrak{p}})$).
We have already chosen a basis $C = (c_1,\ldots,c_{2g})$ of the left hand side.
We take the $\ZZ_p$-basis
$C' = (N_{\reflextype}(x)_p^{-1} b_1, \ldots, N_{\reflextype}(x)_p^{-1} b_{2g})$ of the right hand side and notice
that~$A_p$ transforms one basis to the other
in the sense that $C' = N_{\reflextype}(x)_p^{-1} B = \epsilon(N_{\reflextype}(x))^{-1}_p B
= A_p C$.

In particular, we have $A_p\in\GL_{2g}(\ZZ_p)$. 
As the basis on the left is symplectic for~$N(\mathfrak{a})\xi$ and the one
on the right is symplectic for~$N(x)_p\xi$, we apply
Lemma~\ref{lem:premundane} and
find
$A_p\in\GSp(\QQ_p)$.
As we already had $A_p\in\GL_{2g}(\ZZ_p)$, we conclude
$A_p\in\GSp(\ZZ_p)$.
\end{proof}

\begin{proof}[Proof of Theorem~\ref{thm:general}]
The fact that $f(\tau)$ is in the ray class field modulo~$NF$
is Proposition~\ref{prop:conductor}.
We have $M\in\GSp(\QQ)^+$, $\nu(M) = N(\mathfrak{a})^{-1}$, and $\tau'=M\tau$
by Lemma~\ref{lem:mundanenew}.

It remains to prove that $M$ is invertible modulo $N$
and that $U = (M\ \mathrm{mod}\ N)^{-1}$ is in $\GSp(\ZZ/N\ZZ)$
and satisfies $f(\tau)^{[\mathfrak{a}]} = f^U(M\tau)$.

We have $\matrixofShimura{x}^{-1} = AM$ with
$A\in \GSp(\ZZhat)\times \GSp(\RR)^+$
by Lemma~\ref{lem:defofa}.
This and~\eqref{eq:trivialbutimportant}
imply that $M$ is invertible modulo $N$ and
that the inverse $U$ is $(A \mod N)$.
The ad\`elic reciprocity law (Theorem~\ref{thm:shimrecipad})
tells us $f(\tau)^{[\mathfrak{a}]} = f(\tau)^x = f^{AM}(\tau) = f^A(M\tau)$.
By Corollary~\ref{cor:zzmod}, we find that
$A$ acts on~$f$ as $U = (A \mod N)$ does.
Conclusion: $f(\tau)^{[\mathfrak{a}]} = f^{U}(M\tau)$.
\end{proof}

\begin{proof}[Proof of Theorem~\ref{thm:special}]
Theorem~\ref{thm:special} is a special case of
 Theorem~\ref{thm:general}
as follows.
Pick $C = \mu^{-1} B$ in Theorem~\ref{thm:general}, that is,
$c_i = \mu^{-1} b_i$ for $i=1,\ldots, 2g$,
so $M = \leftmult{1}{C}{B} = \leftmult{\mu^{-1}}{B}{B}$.
Then $M\tau=\tau$ since
multiplication by $\Phi(\mu)$ is a $\CC$-linear isomorphism that transforms
one symplectic basis into the other.
We also have $U = (M\ \mathrm{mod}~N)^{-1} = (\leftmult{\mu}{B}{B}\ \mathrm{mod}~N)$.
\end{proof}

\subsection{Determining the ideal group}
\label{sec:proofidealgroup}

Next, we prove Theorem~\ref{thm:idealgroup},
which states $\mathrm{Gal}(\cmext{N}/\reflexfield) = \IKr{NF}/\cmgp{N}$.
\begin{proof}[Proof of Theorem~\ref{thm:idealgroup}]
Note that Theorem~\ref{thm:special} and Lemma~\ref{lem:epsilonmodstar}
already imply that
$\cmgp{N}$ acts trivially on~$\cmext{N}$.
It remains to prove that if $\mathfrak{a}\in \IKr{NF}$ acts trivially
on $\cmext{N}$, then $\mathfrak{a}\in \cmgp{N}$.
Here without loss of generality the ideal $\mathfrak{a}$
is integral, that is, we have $\mathfrak{a}\subset\mathcal{O}_{K^r}$.

So let $\mathfrak{a}\in \IKr{NF}$ be an integral ideal
with $f(\tau)^{\mathfrak{a}}=f(\tau)$ for all~$f\in\FcalN$.
Let $U$ and $M$ be as in Theorem~\ref{thm:general},
so that for all $f\in\FcalN$, we get
$f(\tau) = f(\tau)^{\mathfrak{a}} = f^{U}(M\tau)$ with $U\in\GSp(\ZZ/N\ZZ)$
and $M\in\GSp(\QQ)^+$ such that~$U = (M\mod NF)^{-1}$.
We claim that without loss of generality, we have $U=1$,
$M\equiv 1\modstar N$ and $M\tau=\tau$.

Proof of the claim:
By taking $f=\zeta_N$, we find $\zeta_N^{\nu(U)}=\zeta_N$,
hence $U\in\Sp(\ZZ/N\ZZ)$.
Then lift $U$ to $\Sp(\ZZ)$, and use the lift to change the
chosen basis
$c_1,\ldots,c_g$ of Theorem~\ref{thm:special}.
We find that without loss of generality, we have $U=1$,
which implies $M\equiv 1\modstar N$.
We now have $f(\tau) = f(M\tau)$ for all $f\in\FcalN$,
and by \cite[(2.5.1)]{shimura-models-I},
this implies $\tau\in\Gamma_N M\tau$, i.e.,
$\tau=\gamma M\tau$ for some $\gamma\in\Gamma_N$.
We use $\gamma$ to change the basis
$c_1,\ldots,c_g$ again,
and conclude also $M\tau=\tau$. This proves the claim.

Let $\mathfrak{c}=\NPhirO(\mathfrak{a})^{-1}\mathfrak{b}$
with basis $C = MB$, leading by Lemma~\ref{lem:mundanenew}(b) to the period
matrix $M\tau$.
We have $M\tau=\tau$, hence there is a polarization-preserving isomorphism
$h:\CC^g/\Phi(\mathfrak{c})\rightarrow
\CC^g/\Phi(\mathfrak{b})=:A$
sending the $i$th element of $C$ to the $i$th element of~$B$.
The identity map on $\CC^g$ induces an isogeny
the other way around,
which scales the polarization by $N(\mathfrak{a})$.
Their composite is some $\mu \in \End(A)=\mathcal{O}$,
which satisfies $\mu^{-1}\mathfrak{b} = \mathfrak{c}$
and $\mu\overline{\mu} = N(\mathfrak{a})\in\QQ$.
This last identity shows that~$\mu$ is coprime to~$F$,
so if we look at the (invertible) coprime-to-$F$ part of
$\mu^{-1}\mathfrak{b} = \mathfrak{c}$,
then we find $\mu\mathcal{O} = \NPhirO(\mathfrak{a})$.

We have $\epsilon(\mu) = M$. 
Lemma~\ref{lem:epsilonmodstar}
therefore shows $\mu\equiv 1\modstar N\mathcal{O}$.
\end{proof}

We have now proven all results from Sections~\ref{ssec:results1}--\ref{ssec:statementbeforelast}.

  \section{Complex conjugation} \label{sec:conjugationproof}

  Next, we prove the results in Section~\ref{ssec:conjugationstatement}.
  \begin{proof}[Proof of Lemma~\ref{lem:complexconjugation}]
  Recall $\mathcal{M}_0 = \realreflex(f(\tau) : f \in \Fcal{1})$,
  and consider the extension $\cmext{1}=\mathcal{M}_0\reflexfield / \mathcal{M}_0$.
  Part (1) of Lemma~\ref{lem:complexconjugation} states
  that this extension has degree 2 if and only if there exist
  $\mathfrak{a}\in I(F)$ and $\mu\in K^\units$
  such that $\NPhirO(\mathfrak{a})\overline{\mathfrak{b}} = \mu \mathfrak{b}$
  and $\mu\overline{\mu} \in \QQ$.

We start by proving the `only if' part, so suppose that
  $\cmext{1}/\mathcal{M}_0$ has degree~$2$.
  The non-trivial automorphism $\gamma_0$ of this extension
  restricts to complex conjugation on~$\reflexfield$,
  so $\gamma:x\mapsto \overline{\gamma_0(x)}$ is an element of $\Gal(\cmext{1}/ \reflexfield)$.
 Suppose that $\tau$ is obtained from $(\mathfrak{b}, \xi)$, and
let $A$ be the corresponding principally polarized abelian variety.
  
  As $\gamma$ and complex conjugation are equal on~$\mathcal{M}_0$,
we get that $\gamma(A)$ and $\overline{A}$ are isomorphic.
  
By \cite[Proposition~3.5.5]{lang-cm}, the abelian variety
  $\overline{A}$ corresponds to $(\overline{\mathfrak{b}}, \xi)$.
  At the same time, the automorphism $\gamma$ corresponds via the Artin map to
  the class of an ideal $\mathfrak{a}$ of $\reflexfield$.
  The isomorphism between $\gamma(A)$ and $\overline{A}$ then gives
  an element $\mu\in K^\units$ such that we have
  $N_{\reflextype}(\mathfrak{a})\overline{\mathfrak{b}} = \mu \mathfrak{b}$
  and $N(\mathfrak{a}) = \mu\overline{\mu}$.
  This proves the `only if' of~(1).
  
  Conversely, if $\mathfrak{a}$ exists, by scaling $\mathfrak{a}$
(and scaling $\mu$ accordingly), we can assume $\mathfrak{a}$ to be coprime
to $NF$. Then take the corresponding
$\gamma\in\mathrm{Gal}(\cmext{1}/\reflexfield)$
and let $\gamma_0:x\rightarrow \overline{\gamma(x)}$,
which is in 
$\mathrm{Gal}(\cmext{1}/\mathcal{M}_0)$
and is non-trivial as it restricts to complex conjugation on~$\reflexfield$.
This prove the `if' part.

  For part~(2), in case $g=1$ and $\mathfrak{b}$ is coprime to $F\mathcal{O}$, we
  simply take
  $\mathfrak{a} = N_{\Phi}(\mathfrak{b}/\overline{\mathfrak{b}})$ and $\mu=1$
  as~$N_{\reflextype}$ is an isomorphism with inverse~$N_{\Phi}$.

  If $g=2$ and $\mathfrak{b}$ is coprime to $F\mathcal{O}$,
  take $\mathfrak{a} = N_{\Phi}(\mathfrak{b}\mathcal{O}_K)$ and $\mu=N(\mathfrak{b})$.
  We have $N_{\reflextype}N_{\Phi}(\mathfrak{b}\mathcal{O}_K)=
  N(\mathfrak{b}) \mathfrak{b}\smash{\overline{\mathfrak{b}}^{-1}}\mathcal{O}_K$
   (see
\cite[(3.3)]{shimura-onabelian} 
or
\cite[(3.2)]{kilicerstreng}), 
  which implies part~(2).

  Finally, if $\mathfrak{b}=\mathcal{O}$, then
  $\overline{\mathfrak{b}}=\overline{\mathcal{O}}=\mathcal{O}$,
  so $\mathfrak{a}=1$ and $\mu=1$ suffice.
  \end{proof}  
 
  \begin{proof}[Proof of Proposition~\ref{prop:complexconjugation}]
  Assume that $\cmext{1}/\mathcal{M}_0$
  is an extension of degree~$2$,
  so there exist $\mathfrak{a}$, $\mu$ and $\gamma_0$ as in the proof of Lemma~\ref{lem:complexconjugation},
  and without loss of generality we have 
  $\mathfrak{a}\in I(NF)$.
  
  Let $f\in\FcalN$ be such that $f(\tau)$ is a class invariant.
  Now $f(\tau)$ is in $\mathcal{M}_0$ if and only if 
  $\gamma_0(f(\tau)) = f(\tau)$ holds,
  that is, if and only if we have $\overline{f(\tau)^{[\mathfrak{a}]}} = f(\tau)$.
    
  The action of complex conjugation is easy to describe.
  For $h\in \FcalN$, note that $h^{i(-1\mod N)}$ is $h$
  with its Fourier coefficients replaced by their
  complex conjugates. Since complex conjugation is continuous
  on~$\CC$, we get
  \begin{equation}\label{eq:cconj}\overline{h(\tau)} =
   h^{i(-1\mod N)}(-\overline{\tau}).
  \end{equation}
  
  Let us look at the action of $[\mathfrak{a}]$ via the reciprocity
  law (Theorem~\ref{thm:general}).
  Let 
  $b_1,\ldots,b_{2g}$ be a symplectic basis of $\mathfrak{b}$ that gives rise to
  to $\tau$
  and consider the symplectic basis
  \[C = (\mu^{-1}\overline{b_1},\ldots,\mu^{-1}\overline{b_g},-\mu^{-1}\overline{b_{g+1}},\ldots,-\mu^{-1}\overline{b_{2g}})\]
  of $\mu^{-1}\overline{\mathfrak{b}} = \NPhirO(\mathfrak{a})^{-1}\mathfrak{b}$
  with respect
  to 
  $\mu\overline{\mu}\xi = N(\mathfrak{a})\xi$,
  which gives rise to the period matrix~$-\overline{\tau}$.
  By Theorem~\ref{thm:general}, we have
$\leftchange{C}{B}\in \GSp(\QQ)^+$,
$U := (\leftchange{C}{B}\ \mathrm{mod}~N)^{-1}\in\GSp(\ZZ/N\ZZ)$,
and $f^{[\mathfrak{a}]} = f^{U}(-\overline{\tau})$.
  
  The basis $C$ differs from $Q$ by multiplying the final $g$
  entries by $-1$,
  so we have $\leftchange{C}{Q} = i(-1)$.
  In particular, we have $\leftchange{Q}{B} = i(-1)\leftchange{C}{B}$,
  hence $V = U i(-1\bmod N)$.
  
  Applying \eqref{eq:cconj} to $h = f^U$, we conclude
  $\overline{f(\tau)^{[\mathfrak{a}]}} = f^{V}(\tau)$,
  so indeed we have $f(\tau)\in\mathcal{M}_0$
  if and only if $f^{V}(\tau) = f(\tau)$.
  
  Finally, suppose that we have $\alpha = f(\tau)\in \mathcal{M}_0$.
  Let $P\in \reflexfield[X]$ be the minimal polynomial of $\alpha$ over~$\reflexfield$.
  Then $\overline{P} = \gamma_0(P)$ is the minimal polynomial of $\gamma_0(\alpha)$ over~$\reflexfield$.
  In the case $\alpha\in\mathcal{M}_0$, we have $\gamma_0(\alpha) = \alpha$, hence $\overline{P}=P$,
  so $P$ has coefficients in~$\realreflex$.
  \end{proof}


\section{Theta constants} \label{sec:theta}
 \newcommand{\prm}{}
 \newcommand{\prmt}{\transpose}
 
 \newcommand{\tinv}{t_{\mathrm{inv}}}
 
For $c_1,c_2\in\QQ^g$, the \emph{theta constant} with characteristic $c_1,c_2$ is
the map $\theta[c_1,c_2] : \HH_g \rightarrow \CC$ given by
      \begin{equation}\label{eq:deftheta}\theta
[ c_1,c_2](\tau)= \sum_{v\in \ZZ^g} \text{exp}(\pi i (v+c_1)\transpose \tau (v+c_1)+2\pi i (v+c_1)\transpose
c_2).
\end{equation}
We often restrict to theta constants with $c_i\in [0,1)^g$,
because we have
\begin{align}\label{eq:thetamodz}\theta[c_1+n_1, c_2+n_2] &= \exp(2\pi i c_1\transpose n_2)\theta[c_1, c_2]& \mbox{for}\quad n_1,n_2\in\ZZ^g.
\end{align}

Theta constants have a very explicit action, as the following
result shows.
The result itself is not surprising,
but the author is unaware of an equally explicit version
in the literature:
directly working for $\GSp$
instead of only $\Sp$
and working with arbitrary coefficient-wise lifts
instead of having to lift to $\Sp(\ZZ)$.

\begin{proposition}\label{prop:theta}
Given $D\in 2\ZZ$ and $c_1,c_2,c_1',c_2'\in D^{-1}\ZZ^{g}$, we have
$$\frac{\theta[c_1,c_2]}{\theta[c_1',c_2']}\in\mathcal{F}_{2D^2}.$$
Moreover, the action of $A\in \GSp(\ZZ/2D^2\ZZ)$ is as follows.
Take lifts
$$B=\tbt{a & b \\ c & d}\in\matrixring{2g}{\ZZ}
\quad \mbox{and}\quad \tinv\in\ZZ$$
of $A$ and $\nu(A)^{-1}$.
Define 
$$\left(\genfrac{}{}{0pt}{0}{d_1^{\prm}}{d_2^{\prm}}\right) =
B\transpose \left(\genfrac{}{}{0pt}{0}{ c_1^{\prm}-\frac{1}{2}\ \tinv\ \mathrm{diag}(cd\transpose)}{
c_2^{\prm} - \frac{1}{2}
                   \ \tinv\ \mathrm{diag}(ab\transpose)}\right)\quad\mbox{and}
                   $$
$$r^{\prm}=\frac{1}{2}(\tinv (dd_1^{\prm}-cd_2^{\prm})\transpose (-bd_1^{\prm}+ad_2^{\prm}+\mathrm{diag}(ab\transpose)) - d_1^{\prmt} d_2^{\prm}),$$
and define $d_1'$, $d_2'$, $r'$ analogously.
Then we have
\begin{equation}
\label{eq:theta}
\left(\frac{\theta[c_1,c_2]}{\theta[c_1',c_2']}\right)^A =
\exp(2\pi i (r-r'))
\frac{\theta[d_1,d_2]}{\theta[d_1',d_2']}.
\end{equation}
\end{proposition}

\begin{remark}
It is known that the field generated by all quotients
as in Proposition~\ref{prop:theta} (for all $D$)
equals the field~$\Fcalall$ (see for example \cite[27.15]{shimura}).
\end{remark}

To prove Proposition~\ref{prop:theta} we use the following 
lemma
giving the action of $\Sp(\ZZ)$.
\begin{lemma}\label{lem:theta}
Given $B\in\Sp(\ZZ)$, there is a holomorphic $\rho =\rho_B: \HH_g\rightarrow \CC^{\units}$
such that
for all $c_1,c_2\in\QQ^g$, we have
$$\theta[c_1,c_2](B\tau) = \rho(\tau) \exp(2\pi i r)\theta[d_1,d_2](\tau),$$
where $d_1$, $d_2$, $r$ are as in the formulas of Proposition~\ref{prop:theta}
with $t=1$.
\end{lemma}
\begin{proof}
This follows with some algebraic manipulation when
substituting our $d$ for the $c$ in Formula~8.6.1
of~\cite{birkenhake-lange}
(see \cite[Lemma~8.4.1(b)]{birkenhake-lange} 
for the definition of $M[d]$).
 \end{proof}
\begin{remark}
	The interested reader could
	see \cite[Exercise 8.11(9)]{birkenhake-lange}
	for more information about~$\rho_B$.
\end{remark}

\begin{proof}[{Proof of Proposition~\ref{prop:theta}}]
	Let $N=2D^2$.
	We start by showing that the right hand side of \eqref{eq:theta} is independent
	of the choices of lifts.
	
	Note that a change of lift changes $B\transpose$
	at most by adding elements of $2D^2\ZZ$ to the entries.
	Similarly, it changes $c_i - \frac{1}{2}\ \tinv\ \mathrm{diag}(\cdots)\in D^{-1}\ZZ^g$
	at most by adding elements of $D^2\ZZ^g$.
	In particular, it changes $d_1$, $d_2$, $d_1'$, and $d_2'$
	at most by adding elements of $2D\ZZ^g$.
	In turn, this means that $r$ changes at most by adding an element of~$\ZZ$.
	Neither change has effect on the right hand side of \eqref{eq:theta}
	by~\eqref{eq:thetamodz}.
	
	Now that we know that \eqref{eq:theta} is independent
	of the chosen lifts, we prove it
	for $A\in\Sp(\ZZ/N\ZZ)$
	by taking a lift in $B\in \Sp(\ZZ)$,
	taking $\tinv = 1$,
	and applying Lemma~\ref{lem:theta}
	to the numerator and denominator,
	where the factors $\rho(\tau)$ cancel.
	
	Next, we show that $f = \theta[c_1,c_2]/\theta[c_1',c_2']$
	is indeed in $\FcalN$.
	First multiply the numerator and denominator
of~$f$ by $\theta[0,0]^7$.
    Then we use Lemma~\ref{lem:theta}
    with $\rho_B(\tau)^{8} = (\det c\tau+d)^4$.
    We have already done all the computations
    required for checking that these modular forms
    are invariant under~$\Gamma_N$.
    As the Fourier coefficients are in $\QQ(\zeta_N)$
    by the definition~\ref{eq:deftheta},
    we find $f\in\FcalN$.

\newcommand{\ch}[1]{#1^0}

Finally, any element $A\in \GSp(\ZZ/N\ZZ)$ can be written as~$A=\ch{A}i(t)$.
with $t = \nu(A)\in\ZZ/N\ZZ$, and $\ch{A}\in\Sp(\ZZ/N\ZZ)$.
Choose lifts $B\in\ZZ^{2g\times 2g}$ of $\ch{A}$ and $\widetilde{t}\in\ZZ$ of~$t$.
Starting from \eqref{eq:theta} for~$\ch{A}$, we compare what happens
when we either multiply $\ch{A}$
by~$i(t)$ from the right, or act on the right hand side
of \eqref{eq:theta} by~$i(t)$.

The latter replaces $\zeta_N$ by $\zeta_N^{\widetilde{t}}$,
which is equivalent (by the definition \eqref{eq:deftheta})
to changing $r$ into $\widetilde{t}r$ and $(d_1,d_2)$ into $(d_1,\widetilde{t}d_2)$.

Writing $B = (a,b;c,d)$,
we get
$A = ((a,b\widetilde{t};c,d\widetilde{t})\bmod{N})$.
It is straightforward to check
that multiplying
$b$ and $d$
by $\widetilde{t}$ and changing $\ch{t}_{\mathrm{inv}} = 1$
into $t_{\mathrm{inv}}$
changes $(d_1,d_2)$ into $(d_1, \widetilde{t}d_2)$ modulo $D^2\ZZ^{2g}$.
In turn, this changes $r$ into $\widetilde{t}r$ modulo $\ZZ$.
By \eqref{eq:thetamodz} this gives the same result
as just changing $r$ into $\widetilde{t}r$ and $(d_1,d_2)$ into $(d_1,\widetilde{t}d_2)$.
\end{proof}

Given a rational function $f\in\FcalN$ that is expressed in terms of theta constants with
characteristics in $D^{-1} \ZZ^{2g}$ with $N\mid 2D^2$,
we can now evaluate the action of $A\in \Sp(\ZZ/N\ZZ)$ on~$f$.
We do not need to lift $A$ to~$\Sp(\ZZ)$,
only to $\GSp(\ZZ/2D^2\ZZ)$, which
is a relatively simple matter of linear algebra over~$\mathbf{F}_p$
for primes $p\mid 2D$.
And in fact, we choose even to avoid that
by applying the reciprocity theorem (Theorem~\ref{thm:general})
directly with $2D^2$ in place of~$N$
(and using $f\in \Fcal{N}\subset \Fcal{2D^2}$).

If $f$ is a quotient of homogeneous polynomials of equal degree in the theta constants,
then we can simply apply the formulas in Proposition~\ref{prop:theta}
directly to the individual theta constants and do not have to write $f$
as a rational function of quotients of the form $\theta[c_1,c_2]/\theta[c_1',c_2']$.
For example, note that we have
\begin{equation}\label{eq:f} f = 
\frac
{\theta[\frac{1}{2},0,0,\frac{1}{2}]}
{\theta[\frac{1}{2},\frac{1}{2},0,0] + \theta[0,0,0,0]}
\quad =\quad \frac
{\frac{\theta[\frac{1}{2},0,0,\frac{1}{2}]}{\theta[c_1',c_2']}}
{\frac{\theta[\frac{1}{2},\frac{1}{2},0,0]}{\theta[c_1',c_2']} + \frac{\theta[0,0,0,0]}{\theta[c_1',c_2']}},
\end{equation}
and the copies of $\exp(-2\pi i r')\theta[d_1',d_2']^{-1}$ in the numerator and denominator cancel in the end anyway.

	We implemented the formulas of Proposition~\ref{prop:theta}
	in \cite{cmcode}
	as \verb!f^A! where $f$ as in \eqref{eq:f} can be constructed
	using 
	\[ \verb!f = ThetaModForm([1/2,0,0,1/2]) / (ThetaModForm([1/2,1/2,0,0]) +!\]
	\[\verb!ThetaModForm([0,0,0,0]))!.\]

  \section{Finding class invariants and minimal polynomials}\label{sec:examples}
  
  In this section, we demonstrate
  how to use the main results
  for finding class invariants.
  We 
  give additional results
  and algorithms as we need them.
  
Given an order~$\mathcal{O}$ in a CM field~$K$
of degree $2g$
  and a primitive CM type~$\Phi$ of~$K$,
  a \emph{class invariant} is a value $f(\tau)$
  with $f\in\Fcalall$, $\tau$ a primitive CM point
  with CM by $\mathcal{O}$ of type $\Phi$, and $f(\tau)\in \cmext{1}$.
For example, if $K$ is quadratic and $\mathcal{O} = \ZZ+\tau\ZZ$,
then $j(\tau)$ is a class invariant, and its minimal polynomial over~$K$
is called the~\emph{Hilbert class polynomial}~$H_{\mathcal{O}}\in\ZZ[X]$.
Weber \cite{weber3} gave class invariants of imaginary quadratic orders
with minimal polynomial that have much smaller coefficients
than~$H_{\mathcal{O}}$
and from which $j(\tau)$ can be recovered.
For CM fields of degree~$2g$, we compare
the height of our
our class invariants with the height of values of
known generators
of~$\Fcal{1}$,
such as $j$ for $g=1$ and
absolute Igusa invariants~\cite{igusa}
for $g=2$.
 
Given~$f\in\FcalN$, we check the inclusion $\reflexfield(f(\tau))\subset\cmext{1}$
(equivalently~$f(\tau)\in\cmext{1}$) using Theorem~\ref{thm:special}.
If~$f$ is sufficiently general, then the inclusion of fields
$\reflexfield(f(\tau))\subset\cmext{1}$ is
an equality,
which can be verified numerically using Theorem~\ref{thm:general}.
The latter theorem also allows us to
numerically determine the minimal polynomial of $f(\tau)$ over~$\reflexfield$.

\subsection{Finding a class invariant}\label{sec:findingclassinv}\label{ssec:detailedexample}

In this example, consider
quotients $f$
of products of theta constants with~$c_1,c_2\in\{0,\frac{1}{2}\}^2$,
that is, $g=2$, $D=2$,~$N=8$.
We also include this example
at the beginning of the file \verb!article.sage! at~\cite{cmcode},
so it could be followed step by step on a computer.
The theta constants
for which~$4c_1^{\phantom{t}} c_2\transpose$ is odd are identically zero,
and we are left with~$10$ so-called 
\emph{even theta constants}, which happen to have Fourier coefficients
in~$\ZZ$.
Following~\cite{dupont}, we use the notation
$\theta[(a,b), (c,d)] =: \theta_{16b + 8 a +4 d + 2c}$
for $a,b,c,d\in\{0,\frac{1}{2}\}$, so 
the even theta constants are $\theta_k$
for $k\in\{0,1,2,3,4,6, 8, 9, 12, 15\}$.

We take the quartic CM field
$K= \QQ(\alpha) = \QQ[X]/(X^4+27X^2+52)$
from \cite[Example~III.3.2]{phdthesis}.
Its real quadratic subfield is $K_0=\QQ(\sqrt{521})$.
Take the CM type~$\Phi = \{\phi : K\rightarrow \CC \mid\phi(\alpha)\in i\RR_{>0}\}$
and let $w=\sqrt{13}\in\RR_{>0}$.
The real quadratic subfield of the reflex field~$\reflexfield$
is~$\QQ(w)$.

We start by finding 
a period matrix $\tau = \tau(\Phi, \mathfrak{b}, \xi,B)$
as in Section~\ref{sec:algperiodmatrices}.
In our case, this yields $\mathfrak{b}=\mathcal{O}$, 
$\xi=2(22411531\alpha^3 + 46779315\alpha)^{-1}$,
and a symplectic basis
\begin{align*}
 B = &\textstyle{\frac{1}{4}}(653 \alpha^{3} + 3414 \alpha^{2} + 1363 \alpha + 7126,
& 401 \alpha^{3} + 2360 \alpha^{2} + 837 \alpha + 4926,\\
&\quad -653 \alpha^{3} + 1306 \alpha^{2} - 1363 \alpha + 2726,
& 2108 \alpha^{2} + 4400).
 \end{align*}

Next, we compute generators of the image of the map
$$\wascalledg : \frac{\IKr{N}\cap\cmgp{1}}{\cmgp{N}}
\longrightarrow  \GSp(\ZZ/N\ZZ)/\bracketsO$$ from
\eqref{eq:wascalledg} in
Section~\ref{sec:statementspecific} using
the command \verb!reciprocity_map_image(tau, 8)! of \cite{cmcode},
that is, using Algorithm~\ref{alg:imageofr}.
This yields a list $R$ of $6$ matrices in $\GSp(\ZZ/8\ZZ)$.

A function $f\in\Fcal{8}$ yields a class invariant if it is fixed by all elements of~$R$.
Let us look at the action on quotients of theta constants
of Proposition~\ref{prop:theta} more closely, starting with $8$th powers
so that the factor $\exp(2\pi i (r-r'))$ vanishes.
This action can be viewed as an action on the numerator
and denominator separately.
So this is an action of $\GSp(\ZZ/8\ZZ)$ on the set of 8th powers of the ten even theta constants.
Under the action of the subgroup generated by~$R$, we compute that
this set is partitioned into~4 orbits:
$\{\theta_0^8, \theta_{1}^8, \theta_6^8\}$,
$\{\theta_2^8, \theta_4^8, \theta_3^8\}$, $\{\theta_8^8,\theta_9^8,\theta_{15}^8\}$,
$\{\theta_{12}^8\}$.

Let us restrict our search for class invariants
to
those functions~$f$ that are products
of powers of the theta constants.
To ensure that the image of~$\wascalledg$ fixes
$f^8$ up to units,
we use whole orbits,
that is, write $$f=
c
(\theta_0\theta_6\theta_1)^j
(\theta_2\theta_3\theta_4)^k
(\theta_8\theta_9\theta_{15})^l\theta_{12}^{m}$$ with 
$c\in\QQ^{\mathrm{ab}}$
and integers $j$, $k$, $l$, $m$ that satisfy $3j+3k+3l+m=0$.

There are various values of $(j,k,l,m)$ that one could try, but we prefer
the minimal polynomial of $f(\tau)$ over $\reflexfield$ to have coefficients in~$\realreflex$,
so we also look at the action
of $V$ from 
Proposition~\ref{prop:complexconjugation}.
It turns out that this action swaps the first two orbits,
so we take~$j=k$.
In fact, we like to use small products of theta constants,
so we leave out these six theta constants, that is, we take $j=k=0$. We then get $3l=-m$, so
with $n=-l$ we get
\[f = c f_0^n\quad\mbox{where}\quad f_0 = \frac{\theta_{12}^3}{\theta_8\theta_9\theta_{15}}.\]
Note that if $8$ divides $n$ and $c\in\QQ$,
then we have $f(\tau)\in \cmext{1}$, but to let $f(\tau)$ have small height, we want
to try smaller values of~$n$.

Explicitly computing the action of $R$ and $V$
on~$f_0$ and $\zeta_8$,
and trying out every $c\in \mu_8$
for $n=1,2,\ldots$,
we find that
$n=2$, $c=\zeta_8^2$ gives a function that
is invariant under $R$ and $V$,
so that we have $f(\tau)\in\mathcal{M}_0$.

The steps above illustrate a general algorithm,
which is also what we followed when
creating the examples mentioned in Section~\ref{ssec:moreexamples} below.

It is however sometimes too restrictive to only consider roots
of unity~$c$, as demonstrated by Sot\'akov\'a~\cite{sotakovamsc}
(even in the case $g=1$).
In Section~\ref{ssec:quotientalgorithm}, we give the higher-dimensional
version of Sot\'akov\'a's ideas for finding the optimal $n$ and~$c$.
For this particular function $f$ it still yields $n=2$ as the smallest valid exponent.

\subsection{Computing the minimal polynomial}
\label{sec:computingminpoly}

So now we have our class invariant $f(\tau)\in\cmext{1}$ and we
would like to compute its minimal polynomial over~$\reflexfield$.
We have $\Gal(\cmext{1}/\reflexfield)
= \IKr{1}/\cmgp{1}$ (Theorem~\ref{thm:idealgroup}).
In general, this group could be computed using
the methods of \cite[Section~4.2]{enge-thome}.
In this particular case, the class number of $K^r$
is odd and the class group of its real quadratic subfield is trivial,
hence (see \cite[Example~I.10.4]{phdthesis})
the Galois group is simply the class group of~$\reflexfield$.

For each of the~$7$ ideal classes of~$\reflexfield$,
we compute~$U$ and~$\tau'$ as in Theorem~\ref{thm:general}.
We make sure that the basis $C$ is such that
$\tau'$ is \emph{reduced} for the action of~$\Sp(\ZZ)$
(see the end of Section~\ref{sec:basis})
so that the theta constants
can be numerically evaluated most efficiently.

Then we compute $f^U$ as in Section~\ref{sec:theta}
and evaluate it numerically at~$\tau'$
to get a root of the minimal polynomial of $f(\tau)$ over~$\reflexfield$.
This yields an approximation of
$$H_f = \prod_{i=1}^{7} (X-f^{U_i}(\tau_i'))\in \realreflex[X],$$
and we recognize its coefficients as elements of $\realreflex\subset\CC$
with the LLL-algorithm as in~\cite[Section~7]{lattices}.
The entire calculation is in the file 
\verb!article.sage! of~\cite{cmcode}.

We find that numerically with high precision, we have
{\footnotesize
\begin{align*}
3^8 101^2 H_f& = 66928761 X^7 + (21911488848 w - 76603728240) X^6 \\
&\quad + (-203318356742784 w + 733099844294784) X^5 \\
&\quad + (-280722122877358080 w + 1012158088965439488) X^4 \\
&\quad + (-2349120383562514432 w + 8469874588158623744) X^3 \\
&\quad + (-78591203121748770816 w + 283364613421131104256) X^2 \\
&\quad + (250917334141632512 w - 904696010264018944) X \\
&\quad - 364471595827200 w + 1312782658043904,
\end{align*}}

\noindent which is significantly smaller than the smallest minimal polynomial
obtained when using
Igusa invariants, even with the small Igusa invariants
from~\cite{runtime}:
\footnotesize
\begin{align*}
& 101^2 H_1 =  10201 X^7\\
&\qquad +  (155205162116358647755w + 559600170220938887110)X^6\\
&\qquad + (152407687697460195175920750535594152550 w \\
&\qquad\quad +    549513732768094956258970636118192859400) X^5\\
&\qquad + \tfrac{1}{2}(2201909580030523730272623848434538048317834513875w\\
&\qquad\quad + 7939097894735431844153019089320973153011210882125)X^4\\
&\qquad + (1047175262927393182849164587480891367594710449395570625 w\\
&\qquad\quad + 3775644104882200832865729346429752069380200097845736875)X^3\\
&\qquad +
\tfrac{1}{2}(90739291480049485513675299110604131111640471324738060%
7234375w\\
&\qquad \quad +
327165168130591119268893142372375309476346120037916993%
8284375)X^2\\
&\qquad + 
    (1501416604965651986004588022297124411339065052590506998%
7454062500w\\
    &\qquad\quad + 
                541343455503671907856059844455869398930835318514053659%
78411062500)X\\
&\qquad + \tfrac{1}{2}(32085417029115132212877701052175189051312077050549053%
7777676328984375w\\
    &\qquad\quad + 
115685616293120067038709321144324285012570966768326545%
9917987279296875).
\end{align*}
\normalsize

As the first polynomial
is so much smaller, we needed a much lower precision
to reconstruct it from a numerical approximation.
As our invariant $f$ is built up from the same
theta constants as the absolute Igusa invariants
(see \cite[Section~8]{runtime}),
it takes the same time to evaluate it
to any given precision, so saving precision in
this way means saving time.

\subsection{More examples}\label{sec:engethomeexample}\label{ssec:moreexamples}

We searched for class invariants with $D=g=2$
for a few more fields.
For each of the fields we tried, the results were similar to
Section~\ref{sec:findingclassinv}:
an easily found product of powers of the ten even
theta constants yielded a class invariant, which 
reduced the precision required for
finding the class polynomials.
We made such examples available online in \verb!article.sage! at~\cite{cmcode}.

We mention one of them in particular.
Andreas Enge and Emmanuel Thom\'e,
when demonstrating their implementation
of a method for computing class polynomials~\cite{enge-thome},
presented at the GeoCrypt 2011 conference
a computation of the Igusa class polynomials of
the maximal order~$\mathcal{O}_K$ of the
field $K=\QQ[X]/(X^4 + 310 X^2 + 17644)$
of class number~$3948$.

Following the steps of Section~\ref{sec:findingclassinv},
we found that the functions
$$t=\frac{\theta_0\theta_8}{\theta_4\theta_{12}}\in \mathcal{F}_8,\quad
u=\left(\frac{\theta_2\theta_8}{\theta_6\theta_{12}}\right)^2\in\mathcal{F}_2,\quad  v = \left(\frac{\theta_0\theta_2}{\theta_4\theta_6}\right)^2\in\mathcal{F}_2$$
are
class invariants for a certain $\tau$
with CM by~$\mathcal{O}_K$.

These invariants have the additional advantage that
$(t^2, u, v)$ are \emph{Rosenhain invariants},
meaning that
the abelian variety corresponding to $\tau$
is the Jacobian of
$$C: y^2 = x(x-1)(x-t(\tau)^2)(x-u(\tau))(x-v(\tau)).$$
In particular, they are especially useful
for constructing curves as
in Section~\ref{ssec:cmmethod}.

We believe for two reasons that these class invariants have much smaller height
than the Igusa invariants. 
First, this is what happened in our other examples with quotients
of small products of theta functions,
and second it is claimed in~\cite{cdly} that Rosenhain invariants typically have much smaller height
than Igusa invariants.
As a result we expect that
these invariants would have significantly sped up the computation
for the example of Enge and Thomé.

\subsection{More general constants}\label{ssec:quotientalgorithm}

  In Section~\ref{ssec:detailedexample} we described a general procedure
  for finding class invariants of the form $c \prod_{i} \theta[c_i]^{e_i}(\tau)$
  with roots of unity~$c$.
  Sot\'akov\'a in her MSc thesis~\cite{sotakovamsc}
  showed how to do the same with arbitrary elements
  $c\in \QQ(\zeta_N)^\units$ that are not necessarily roots of
  unity.
  Her ideas come down to the use of
  the inflation-restriction sequence from group cohomology
  and Hilbert's theorem 90,
  and amount to the following result.
  
  \begin{proposition}[{cf.~\cite[Section~5.4.3]{sotakovamsc}}]
  	Given 
  	a subgroup $G\subset \GSp(\ZZ/N\ZZ)$,
  	let $\FcalN^G$ be the fixed subfield.

  	Let $H
  	\subset G$
  	and $C\subset (\ZZ/N\ZZ)^\units$
  	be the kernel
  	and image
  	of $\nu : G\rightarrow (\ZZ/N\ZZ)^\units$.
	Then we have
	\begin{equation}\label{eq:fcalG}
	\{ f \in \FcalN^\units : \exists_{c\in \QQ(\zeta_N)^\units}\ cf \in \FcalN^G\}
	= \left\{f \in \FcalN^\units :\begin{array}{l}\forall_{A\in G}\ \ f^A/f\in\QQ(\zeta_N)^\times\\
  		\text{and}\ \ \forall_{A\in H}\ f^A = f\end{array}\right\}.
  	\end{equation}
  	
    Moreover, for every element $f$ of the right hand side
    we can find $c$ with $cf\in \FcalN^G$ as follows:
    \begin{enumerate}
  				\item let \begin{align*}\phi : \quad C  = \nu(G) &\rightarrow \QQ(\zeta_N)^\units\\
  					\nu(A)&\mapsto f^A/f,
  					\end{align*}
  				\item take any $y\in\QQ(\zeta_N)$ such that
$$c := \sum_{k \in C}  \phi(k) y^{\sigma(k)}\not=0,$$
where $\zeta_N^{\sigma(k)} = \zeta_N^k$.
   \end{enumerate}
  \end{proposition}
\begin{remark}Two typical groups $G$ we take are as follows.
	Let $\tau$ be a CM period matrix and let $G'$ be the preimage
    $\GSp(\ZZ/N\ZZ)$ of the image of $\wascalledg$
    as in~\eqref{eq:wascalledg}.

We can take $G = G'$,
and then $f\in\FcalN^G$ implies
$f(\tau)\in \cmext{1}$
by Theorem~\ref{thm:special}.

Alternatively, we take $G = \langle G', V\rangle$ to be the subgroup of
$\GSp(\ZZ/N\ZZ)$ generated by
$G'$ and the complex conjugation matrix $V$
of Proposition~\ref{prop:complexconjugation}.
In that case $f\in\FcalN^G$ implies
$f(\tau)\in \mathcal{M}_0$,
so that the minimal polynomial of $f(\tau)$
over $K^r$ has coefficients in $K_0^r$.
\end{remark}

\begin{remark}\label{rem:algorithmconstant}
In practical computations, we consider
a free abelian subgroup $F\subset \FcalN^\units/\QQ(\zeta_N)^\units$
generated by finitely many theta constants and stable
under the action of $\GSp(\ZZ/N\ZZ)$.
Then the conditions on the right hand side of \eqref{eq:fcalG}
come down to linear equations in the $\ZZ$-module $F$,
and hence finding the intersection of $F$ with the right
hand side of \eqref{eq:fcalG}
comes down to linear algebra over~$\ZZ$.
\end{remark}

\begin{proof}
	The left hand side of \eqref{eq:fcalG}
	is contained in the right
	because its elements $f$ satisfy $f^A = (cf)^A/c^A = cf / c^A = (c/c^A) f$
	for all $A\in G$ with $c^A=c$ if $A\in H$.
    
    For the reverse inclusion, it suffices to show that the procedure
    of steps (1) and~(2)
    is correct.
    Note that the map $\phi$ is well-defined precisely
    when $f$ is in the right hand side of~\eqref{eq:fcalG}.
    This map $\phi$ is a $1$-cocycle for the group $C\subset (\ZZ/N\ZZ)^\units
    =\Gal(\QQ(\zeta_N)/\QQ)$ and the $C$-module $\QQ(\zeta_N)^\units$,
    that is, we have $\phi(k\ell) = \phi(k)^{\sigma(\ell)} \phi(\ell)$.
    Such a cocycle is a coboundary by Hilbert's theorem 90
    ($H^1(\Gal(E/F), E^\units)=0$ with $E = \QQ(\zeta_N)$
    and $F = E^C$).
    In fact, the proof of Hilbert's theorem 90 comes down to $c$
    being non-zero for some~$y$
    together with the direct verification of
    the identity $c^{\sigma(\ell)} = \phi(\ell)^{-1} c$.
    This gives $(cf)^A = \phi(\nu(A))^{-1}c\phi(\nu(A)) f = cf$,
    hence $cf\in\FcalN^G$.
\end{proof}
As examples where this procedure finds small class
invariants where $c$ is not a root of unity,
see \cite[Sections 6.2 and~6.3]{sotakovamsc}.

  \section{Applications} \label{sec:applications}

\subsection{Class fields}

Hilbert class fields of number fields can be computed using Kummer
theory \cite{cohen2, magma}, but that requires extending
the base field with auxiliary roots of unity, which make such computations
too costly for larger examples.
Complex multiplication yields a more efficient
way to compute the Hilbert class field
if the base field is imaginary quadratic~\cite{pari,stevenhagen-computationalcft}
or quartic CM~\cite{asuncionthesis}.
Class invariants yield a further speed-up
by lowering the required precision.

\subsection{Curves of genus two with prescribed Frobenius}
\label{ssec:cmmethod}

In this section we show how class invariants
give a practical improvement to the CM method
for constructing curves of genus two.
We start with a sketch of the CM method without class invariants
(\ref{sssec:cm1}).
Then we recall how class invariants are used in genus one
(\ref{sssec:cm2}).
Finally we explain how class invariants give
an improvement
in genus two (\ref{sssec:cm3}).

\subsubsection{The CM method}
\label{sssec:cm1}

We would like to construct a 
$g$-dimensional abelian variety over a finite field
with a prescribed characteristic polynomial
$f$ of the Frobenius endomorphism~$\pi$.
Indeed, when choosing $f$ appropriately,
this yields an abelian variety
with a prescribed number of points, or
with good cryptographic properties~\cite{HEHCC,freeman-stevenhagen-streng,spallek}.

The idea of the CM method is to take
an abelian variety $\widetilde{A}$ in characteristic zero with a nice
endomorphism ring~$\mathcal{O}$, and reduce it
modulo a prime.
The endomorphism ring of the reduction $A$
will contain both $\pi$ and~$\mathcal{O}$.
A `lack of space' in $\End(A)$
then relates $\pi$ to $\mathcal{O}$,
giving us the control that we need. 

In more detail, assuming for simplicity that~$f$ is an irreducible Weil polynomial
of degree~$2g$, this works
as follows.
The field $K=\QQ[X]/(f)$ is a CM field
of degree $2g$, the constant coefficient
$f(0) = p^{gm}$ is a prime power,
and the root $\pi_0=(X\ \mathrm{mod}\ f)$ is a
Weil $p^{m}$-number, that is, satisfies
$\pi_0\overline{\pi_0}=p^m$.
Let $\widetilde{A}$ be an abelian variety over a number field $k$
with $\mathrm{End}(\widetilde{A}_{\overline{k}})\cong \mathcal{O}_K$
of CM type~$\Phi$.
Assume $k\supset \reflexfield$, or equivalently,
that the endomorphisms of $\widetilde{A}$ over $\overline{k}$ are defined over~$k$.
Let $\mathfrak{P}/p$ be a prime
of~$k$.
Suppose  that$\widetilde{A}$ has good reduction at $\mathfrak{P}$ and let $A$
be the reduction.
Let $\pi\in\mathrm{End}(A)$ be the Frobenius
endomorphism of~$A$.
Reduction
modulo~$\mathfrak{P}$ gives an embedding $\mathcal{O}_K=\mathrm{End}(\widetilde{A})\subset \mathrm{End}(A)$
and we have the following result.
\begin{theorem}[{Shimura-Taniyama
formula~\cite[Thm.1 in \S13]{shimura-taniyama}}]\label{thm:shimura-taniyama}
The endomorphism $\pi$ is an element of the ring
$\mathcal{O}_K\subset\mathrm{End}(A)$
and generates the ideal
$N_{\reflextype}(N_{k/\reflexfield}(\mathfrak{P}))$
of~$\mathcal{O}_K$.
\end{theorem}
This, together with the fact $\pi\overline{\pi}= \# (\mathcal{O}_k/\mathfrak{P})^g$
determines $\pi$ up to roots of unity.
In fact, by taking $k$ to be minimal, we get $\pi=\pi_0$ up to roots of unity,
that is, up to twists of~$A$.

This CM method can be made to be practical for
at least $g=1$ \cite{atkin-morain, bbel, sutherlandCRT},
$g=2$ \cite{spallek, vanwamelen, dupont, runtime}, and
$g=3$ \cite{weng-g3, koike-weng, bilv, lario-somoza, KLLRSS},
as well as for a certain class of curves with $g=5$ \cite{somoza-thesis}.

In all practical situations, 
one does not write down defining equations for
the characteristic-zero abelian variety~$\widetilde{A}$,
but only evaluates certain modular functions at~$\widetilde{A}$.
For example, for $g=1$, we take the $j$-invariant
and for $g=2$, we take a triple
of absolute Igusa invariants $i_1$, $i_2$, $i_3$.

In the case $g=1$, the elliptic curve $A$ can
be reconstructed from $j(A) = (j(\widetilde{A})\mod \mathfrak{P})$ by a 
textbook formula.
In the case $g=2$, for generic values of the Igusa invariants
modulo $\mathfrak{P}$, one can reconstruct $\widetilde{A}$
as the Jacobian of a hyperelliptic curve using
 Mestre's algorithm~\cite{mestre}. Similar constructions
 are used 
 for $g=3$ and $g=5$.

In the CM method for $g=1$, the value $j(\tau)$ is represented by
its minimal polynomial, the Hilbert class polynomial.
We reduce $j(\tau)$ modulo a prime by reducing
the Hilbert class polynomial modulo~$p$ and
taking a root of that in~$\overline{\FF_p}$.

In the case $g\geq 2$, we take a minimal polynomial~$H_{i_1}$
of the
first invariant $i_1(\widetilde{A})$ over~$\reflexfield$, and we represent
$i_2, \ldots, i_d$ by polynomials
$$ \widehat{H}_{i_1,i_n} = \sum_\gamma i_n(\widetilde{A})^\gamma \prod_{\sigma}
(X - i_n(\widetilde{A})^\sigma) \in \realreflex[X],$$
where sum and product range over $\mathrm{Gal}(\cmext{1}/\reflexfield)$
(see~\cite{ghkrw-2adic}).
Reducing~$H_{i_1}$ modulo a prime $\mathfrak{p}_0$ of $\realreflex$
and taking any root is equivalent to reducing $i_1(\widetilde{A})$ modulo
a prime over~$\mathfrak{p}_0$.
We can then find $i_2(A),\ldots, i_d(A)$
by computing
$$i_n(A) = \frac{\widehat{H}_{i_1,i_n}(i_1(A))}{H_{i_1}'(i_1(A))}$$
if $p$ is sufficiently large.

This is how the CM method works, and now we would like
to use class invariants for efficiency.

\subsubsection{Class invariants for genus one}
\label{sssec:cm2}

We now summarise the (standard) way
in which class invariants
are used in  the CM method in the case $g=1$.
Let $f\in\FcalN$ be a non-constant function
and let $\Phi_{f,j}(X,Y)\in\QQ[X,Y]$
be such that $\Phi_{f,j}(j, Y)\in \QQ(j)[Y]$ is a
(not necessarily monic)
minimal polynomial of
$f$ over $\Fcal{1}=\QQ(j)$.
Then we have $\Phi_{f,j}(j(\tau), f(\tau)) = 0$.
So given $f(\tau)$, we can
find $j(\tau)$ by solving for $X$ in $\Phi_{f,j}(X, f(\tau))=0$.

For the CM method, we compute the polynomial $\Phi_{f,j}$
and the minimal polynomial $H_f$
of a class invariant $f(\tau)$.
Here $\Phi_{f,j}$ can be reused as it depends only on~$f$,
and $H_f$ is much smaller than
the Hilbert class polynomial~$H_j$, hence needs less precision.
We compute $f(\tau)$ modulo a prime over $\mathfrak{P}$
by taking a root $f_0$ of $H_f$ modulo~$p$.
Then we solve for $X$ in 
$\Phi_{f,j}(X, f_0)=0$ to get $j(A)$.

\subsubsection{Class invariants in general}
\label{sssec:cm3}

For general~$g$, we give three methods for using class invariants.

\textbf{Using modular polynomials as in $g=1$.}
For $g\geq 2$, modular polynomials are much harder
to compute \cite{gruenewald-thesis, broker_lauter_modpol, martindale-modpol, milio-robert},
and the higher-dimensional analogue of solving
$\Phi_{f,j}(X, f_0)=0$ involves
Gr\"obner bases.
But for some choices of invariants this may be doable.

\textbf{A modular interpretation of the class invariants.}
Some class invariants themselves give rise to models
of curves or abelian varieties in a direct way, without the need of
invariants from $\Fcal{1}$.
For example, the modular functions $t$, $u$, $v\in\Fcal{8}$
of Section~\ref{sec:engethomeexample}
give rise to the curve
$y^2 = x(x-1)(x-t^2)(x-u)(x-v)$
without the intermediate step of Igusa invariants.

\textbf{Numerically expressing invariants in terms of the class invariant.}
Suppose that $i_1,\ldots, i_d$ are the invariants we need in order to
construct our curve or abelian variety.
We numerically compute $H_f$ and
$$ \widehat{H}_{f,i_n} = \sum_\gamma i_n(\widetilde{A})^\gamma \prod_{\sigma}
(X - f(\widetilde{A})^\sigma) \in \reflexfield[X].$$
These polynomials are in $\realreflex[X]$ if the conditions of Proposition~\ref{prop:complexconjugation}
are satisfied.
We find $f_0$ as a root of $H_d$ modulo~$p$,
and compute $i_n(A)$ for $n$ from it
by the formula
$$i_n(A) = \frac{\widehat{H}_{f,i_n}(f_0)}{H_{f}'(f_0)}.$$

We do need to compute $d+1$ polynomials instead of~$d$,
compared to when only using $i_1,\ldots,i_d$, but
as the size is dominated by the first invariant, which is now $f$
instead of $i_1$, the total size of the polynomials still goes down.

For the example from Sections~\ref{ssec:detailedexample}--\ref{sec:computingminpoly},
we computed the polynomials
$H_{f}$ and $\widehat{H}_{f,i_n}$ and made them available online
(close to line 200 of the file \verb!article.sage! of~\cite{cmcode}).
These four polynomials together take up 15\% less space than
the three polynomials $H_{i_1}$ and $\widehat{H}_{i_1,i_n}$.
More importantly, the largest coefficient (which determines the precision
at which theta constants need to be evaluated, the dominant step
in the computation)
is 40\% smaller.

  \bibliographystyle{plain}
\bibliography{bib,bib2}

\def\cprime{$'$}
\begin{thebibliography}{10}

\bibitem{jaredpaper1}
Jared Asuncion.
\newblock Computing the {H}ilbert class fields of quartic {CM} fields using
  complex multiplication.
\newblock preprint, \href{https://arxiv.org/abs/2104.13639}{arXiv:2104.13639},
  2021.

\bibitem{asuncionthesis}
Jared Asuncion.
\newblock {\em Complex multiplication constructions of abelian extensions of
  quartic fields}.
\newblock PhD thesis, Universit\'e de Bordeaux and Universiteit Leiden, 2022.
\newblock \url{https://hdl.handle.net/1887/3304503}.

\bibitem{atkin-morain}
A.~Oliver~L. Atkin and Fran{\c{c}}ois Morain.
\newblock Elliptic curves and primality proving.
\newblock {\em Math. Comp.}, 61(203):29--68, 1993.
\newblock \url{http://www.inria.fr/rrrt/rr-1256.html}.

\bibitem{bilv}
Jennifer~S. Balakrishnan, Sorina Ionica, Kristin Lauter, and Christelle
  Vincent.
\newblock Constructing genus-3 hyperelliptic {J}acobians with {CM}.
\newblock {\em LMS J. Comput. Math.}, 19(suppl. A):283--300, 2016.

\bibitem{bbel}
Juliana Belding, Reinier Br{\"o}ker, Andreas Enge, and Kristin Lauter.
\newblock Computing {H}ilbert class polynomials.
\newblock In {\em Algorithmic Number Theory -- ANTS-VIII (Banff, 2008)}, LNCS
  5011, pages 282--295. Springer, 2008.

\bibitem{birkenhake-lange}
Christina Birkenhake and Herbert Lange.
\newblock {\em Complex abelian varieties}, volume 302 of {\em Grundlehren der
  mathematischen Wissenschaften}.
\newblock Springer, second edition, 2004.

\bibitem{magma}
Wieb Bosma, John Cannon, and Catherine Playoust.
\newblock The {M}agma algebra system {I}: {T}he user language.
\newblock {\em J. Symbolic Comput.}, 24(3-4):235--265, 1997.
\newblock Computational algebra and number theory (London, 1993).

\bibitem{broker_lauter_modpol}
Reinier Bröker and Kristin Lauter.
\newblock Modular polynomials for genus 2.
\newblock {\em LMS Journal of Computation and Mathematics}, 12:326–339, 2009.

\bibitem{cohen}
Henri Cohen.
\newblock {\em A Course in Computational Algebraic Number Theory}, volume 138
  of {\em Graduate Texts in Mathematics}.
\newblock Springer, 1993.

\bibitem{cohen2}
Henri Cohen.
\newblock {\em Advanced topics in computational number theory}, volume 193 of
  {\em Graduate Texts in Mathematics}.
\newblock Springer-Verlag, New York, 2000.

\bibitem{HEHCC}
Henri Cohen, Gerhard Frey, Roberto Avanzi, Christophe Doche, Tanja Lange, Kim
  Nguyen, and Frederik Vercauteren, editors.
\newblock {\em Handbook of elliptic and hyperelliptic curve cryptography}.
\newblock Chapman \& Hall/CRC, Boca Raton, FL, 2006.

\bibitem{stevenhagen-computationalcft}
Henri Cohen and Peter Stevenhagen.
\newblock Computational class field theory.
\newblock In J.~Buhler and P.~Stevenhagen, editors, {\em Surveys in Algorithmic
  Number Theory}, volume~44 of {\em MSRI Publications}, pages 497 -- 534.
  Cambridge University Press, 2008.

\bibitem{cdly}
Craig Costello, Alyson Deines-Schartz, Kristin Lauter, and Tonghai Yang.
\newblock Constructing abelian surfaces for cryptography via {R}osenhain
  invariants.
\newblock {\em LMS J. Comput. Math.}, 17(suppl. A):157--180, 2014.

\bibitem{DHBHS}
Bernard Deconinck, Matthias Heil, Alexander Bobenko, Mark van Hoeij, and Marcus
  Schmies.
\newblock Computing {R}iemann theta functions.
\newblock {\em Math. Comp.}, 73(247):1417--1442, 2004.

\bibitem{dupont}
R{\'e}gis Dupont.
\newblock {\em Moyenne arithm{\'e}tico-g{\'e}om{\'e}trique, suites de
  {B}orchardt et applications}.
\newblock PhD thesis, {\'E}cole Polytechnique, 2006.
\newblock
  \url{http://www.lix.polytechnique.fr/Labo/Regis.Dupont/these_soutenance.pdf}.

\bibitem{enge-cm}
Andreas Enge.
\newblock {C}{M}.
\newblock software available at \url{http://www.multiprecision.org/cm/}.

\bibitem{enge-morain}
Andreas Enge and Fran{\c{c}}ois Morain.
\newblock Fast decomposition of polynomials with known {G}alois group.
\newblock In {\em Applied algebra, algebraic algorithms and error-correcting
  codes ({T}oulouse)}, LNCS 2643, pages 254--264. Springer, 2003.

\bibitem{enge-streng}
Andreas Enge and Marco Streng.
\newblock Schertz style class invariants for genus two, 2016.
\newblock preprint, \href{https://arxiv.org/abs/1610.04505}{arXiv:1610.04505}.

\bibitem{enge-thome}
Andreas Enge and Emmanuel Thom\'{e}.
\newblock Computing class polynomials for abelian surfaces.
\newblock {\em Exp. Math.}, 23(2):129--145, 2014.

\bibitem{freeman-stevenhagen-streng}
David Freeman, Peter Stevenhagen, and Marco Streng.
\newblock Abelian varieties with prescribed embedding degree.
\newblock In A.~J. van~der Poorten and A.~Stein, editors, {\em ANTS}, volume
  5011 of {\em Lecture Notes in Computer Science}, pages 60--73. Springer,
  2008.

\bibitem{ghkrw-2adic}
Pierrick Gaudry, Thomas Houtmann, David Kohel, Christophe Ritzenthaler, and
  Annegret Weng.
\newblock The 2-adic {CM} method for genus 2 curves with application to
  cryptography.
\newblock In {\em Advances in Cryptology -- ASIACRYPT 2006}, LNCS 4284, pages
  114--129. Springer, 2006.

\bibitem{MR1730432}
Alice Gee.
\newblock Class invariants by {S}himura's reciprocity law.
\newblock {\em J. Th\'eor. Nombres Bordeaux}, 11(1):45--72, 1999.
\newblock Les XX{\`e}mes Journ{\'e}es Arithm{\'e}tiques (Limoges, 1997).

\bibitem{gee-stevenhagen}
Alice Gee and Peter Stevenhagen.
\newblock Generating class fields using {S}himura reciprocity.
\newblock In {\em Algorithmic number theory ({P}ortland, {OR}, 1998)}, volume
  1423 of {\em Lecture Notes in Comput. Sci.}, pages 441--453. Springer,
  Berlin, 1998.

\bibitem{gruenewald-thesis}
David Gruenewald.
\newblock {\em Explicit Algorithms for Humbert Surfaces}.
\newblock PhD thesis, University of Sidney, 2009.
\newblock (3,3) modular polynomial at
  \url{http://www.maths.usyd.edu.au/u/davidg/thesis.html}.

\bibitem{hertogh}
Math{\'e} Hertogh.
\newblock Computing with ad{\'e}les and id{\'e}les.
\newblock MSc thesis, Universiteit Leiden,
  \url{https://hdl.handle.net/1887/3249353}, 2021.

\bibitem{hertoghcode}
Math{\'e} Hertogh.
\newblock Computing with ad{\'e}les and id{\'e}les.
\newblock SageMath code, \url{https://github.com/mathehertogh/adeles}, 2021.

\bibitem{igusa}
Jun-Ichi Igusa.
\newblock Arithmetic variety of moduli for genus two.
\newblock {\em Annals of Mathematics}, 72(3):612--649, 1960.

\bibitem{kilicerstreng}
P\i~nar K\i~l\i \c{c}er and Marco Streng.
\newblock The {CM} class number one problem for curves of genus 2.
\newblock {\em Res. Number Theory}, 9(1):Paper No. 15, 29, 2023.

\bibitem{Genusthreereduction}
P\i{}nar K\i{}l\i{}\c{c}er.
\newblock Reduction of period matrices in genus 3.
\newblock
  \url{https://bitbucket.org/pkilicer/period-matrices-for-genus-3-cm-curves/}.

\bibitem{KLLRSS}
P\i{}nar K\i{}l\i{}\c{c}er, Hugo Labrande, Reynald Lercier, Christophe
  Ritzenthaler, Jeroen Sijsling, and Marco Streng.
\newblock Plane quartics over {$\mathbb{Q}$} with complex multiplication.
\newblock {\em Acta Arith.}, 185(2):127--156, 2018.

\bibitem{koecher}
Max Koecher.
\newblock Zur {T}heorie der {M}odulformen {$n$}-ten {G}rades. {I}.
\newblock {\em Math. Z.}, 59:399--416, 1954.

\bibitem{koike-weng}
Kenji Koike and Annegret Weng.
\newblock Construction of {CM} {P}icard curves.
\newblock {\em Mathematics of Computation}, 74:499--518, 2004.

\bibitem{lang-cm}
Serge Lang.
\newblock {\em Complex Multiplication}, volume 255 of {\em Grundlehren der
  mathematischen Wissenschaften}.
\newblock Springer, 1983.

\bibitem{lario-somoza}
Joan-C. Lario and Anna Somoza.
\newblock An inverse {J}acobian algorithm for {P}icard curves (with an appendix
  by {C}hristelle {V}incent).
\newblock {\em Res. Number Theory}, 7(2):32, 2021.
\newblock \href{https://arxiv.org/abs/1611.02582}{arXiv:1611.02582}.

\bibitem{lattices}
Hendrik~W. Lenstra, Jr.
\newblock Lattices.
\newblock In J.~Buhler and P.~Stevenhagen, editors, {\em Surveys in Algorithmic
  Number Theory}, volume~44 of {\em MSRI Publications}, pages 127 -- 181.
  Cambridge, 2008.

\bibitem{martindale-modpol}
Chloe Martindale.
\newblock Hilbert modular polynomials.
\newblock {\em J. Number Theory}, 213:464--498, 2020.

\bibitem{mestre}
Jean-Fran{\c{c}}ois Mestre.
\newblock Construction de courbes de genre {$2$} {\`{a}} partir de leurs
  modules.
\newblock In {\em Effective methods in algebraic geometry ({C}astiglioncello,
  1990)}, volume~94 of {\em Progr. Math.}, pages 313--334. Birkh\"auser, 1991.

\bibitem{milio-robert}
Enea Milio and Damien Robert.
\newblock Modular polynomials on {H}ilbert surfaces.
\newblock {\em J. Number Theory}, 216:403--459, 2020.
\newblock \url{https://hal.archives-ouvertes.fr/hal-01520262v2}.

\bibitem{newman}
Morris Newman.
\newblock {\em Integral matrices}.
\newblock Academic Press, 1972.
\newblock Pure and Applied Mathematics, Vol. 45.

\bibitem{schertz}
Reinhard Schertz.
\newblock Weber's class invariants revisited.
\newblock {\em Journal de Th\'eorie des Nombres de Bordeaux}, 14(1):325--343,
  2002.

\bibitem{shimura-models-I}
Goro Shimura.
\newblock On canonical models of bounded symmetric domains {I}.
\newblock {\em Ann of Math}, 91:144--222, 1970.

\bibitem{shimura-models-II}
Goro Shimura.
\newblock On canonical models of bounded symmetric domains {II}.
\newblock {\em Ann of Math}, 92:528--549, 1970.

\bibitem{shimura-arithmetic}
Goro Shimura.
\newblock On some arithmetic properties of modular forms of one and several
  variables.
\newblock {\em Ann. of Math. (2)}, 102(3):491--515, 1975.

\bibitem{shimura-fourier}
Goro Shimura.
\newblock On the {F}ourier coefficients of modular forms of several variables.
\newblock {\em G{\"o}ttingen Nachr. Akad. Wiss.}, pages 261--268, 1975.

\bibitem{shimura-theta-cm}
Goro Shimura.
\newblock Theta functions with complex multiplication.
\newblock {\em Duke Mathematical Journal}, (4):673--696, 1976.

\bibitem{shimura-onabelian}
Goro Shimura.
\newblock On abelian varieties with complex multiplication.
\newblock {\em Proc. London Math. Soc. (3)}, 34(1):65--86, 1977.

\bibitem{shimura-reciprocity-theta}
Goro Shimura.
\newblock On certain reciprocity-laws for theta functions and modular forms.
\newblock {\em Acta Math.}, 141(1-2):35--71, 1978.

\bibitem{shimuraAF}
Goro Shimura.
\newblock {\em Introduction to the Arithmetic Theory of Automorphic Functions}.
\newblock Princeton University Press, 1994.

\bibitem{shimura}
Goro Shimura.
\newblock {\em Abelian Varieties with Complex Multiplication and Modular
  Functions}.
\newblock Princeton University Press, 1998.
\newblock Sections 1--16 essentially appeared before in
  \cite{shimura-taniyama}.

\bibitem{shimura-taniyama}
Goro Shimura and Yutaka Taniyama.
\newblock {\em Complex multiplication of abelian varieties and its applications
  to number theory}, volume~6 of {\em Publications of the Mathematical Society
  of Japan}.
\newblock 1961.

\bibitem{somoza-thesis}
Anna Somoza.
\newblock PhD thesis, Universitat Polit{\`e}cnica de Catalunya and Universiteit
  Leiden, 2019.
\newblock Inverse Jacobian and related topics for certain superelliptic curves.

\bibitem{sotakovamsc}
Jana Sot\'akov\'a.
\newblock Eta quotients of class fields of imaginary quadratic fields.
\newblock MSc thesis, Universiteit Leiden and Universit\"at Regensburg,
  \url{https://hdl.handle.net/1887/3597051}, 2017.

\bibitem{spallek}
Anne-Monika Spallek.
\newblock {\em Kurven vom {G}eschlecht {$2$} und ihre {A}nwendung in
  {P}ublic-{K}ey-{K}ryptosystemen}.
\newblock PhD thesis, Institut f{\"u}r Experimentelle Mathematik,
  Universit{\"a}t GH Essen, 1994.
\newblock {\url{http://www.iem.uni-due.de/zahlentheorie/AES-KG2.pdf}}.

\bibitem{MR563924}
Harold~M. Stark.
\newblock {$L$}-functions at {$s=1$}. {IV}. {F}irst derivatives at {$s=0$}.
\newblock {\em Adv. in Math.}, 35(3):197--235, 1980.

\bibitem{sage}
William~A. Stein et~al.
\newblock {\em {S}age {M}athematics {S}oftware ({V}ersion 8.6)}.
\newblock The Sage Development Team, 2019.
\newblock \url{http://www.sagemath.org}.

\bibitem{phdthesis}
Marco Streng.
\newblock {\em Complex multiplication of abelian surfaces}.
\newblock PhD thesis, Universiteit Leiden, 2010.
\newblock \url{http://hdl.handle.net/1887/15572}.

\bibitem{cmcode}
Marco Streng.
\newblock Recip, 2011--2023.
\newblock {RE}pository of {C}omplex mult{IP}lication {S}age{M}ath code,
  formerly package for using {S}himura's {RECIP}rocity law, version TODO
  \url{https://bitbucket.org/mstreng/recip/}.

\bibitem{runtime}
Marco Streng.
\newblock Computing {I}gusa class polynomials.
\newblock {\em Math. Comp.}, 83:275--309, 2014.
\newblock \href{http://arxiv.org/abs/0903.4766}{arXiv:0903.4766}.

\bibitem{sutherlandCRT}
Andrew~V. Sutherland.
\newblock Computing {H}ilbert class polynomials with the {C}hinese remainder
  theorem.
\newblock {\em Math. Comp.}, 80(273):501--538, 2011.

\bibitem{sutherlandclassinv}
Andrew~V. Sutherland.
\newblock Accelerating the {CM} method.
\newblock {\em LMS J. Comput. Math.}, 15:172--204, 2012.

\bibitem{pari}
{The PARI~Group}, Bordeaux.
\newblock {\em {PARI/GP, version {\tt 2.11.1}}}, 2018.
\newblock available from \url{http://pari.math.u-bordeaux.fr/}.

\bibitem{vanwamelen}
Paul van Wamelen.
\newblock Examples of genus two {CM} curves defined over the rationals.
\newblock {\em Math. Comp.}, 68(225):307--320, 1999.

\bibitem{weber3}
Heinrich Weber.
\newblock {\em Algebraische Zahlen}, volume~3 of {\em Lehrbuch der Algebra}.
\newblock Friedrich Vieweg, 1908.

\bibitem{weng-g3}
Annegret Weng.
\newblock Hyperelliptic {CM}-curves of genus 3.
\newblock {\em Journal of the Ramanujan Mathematical Society}, 16(4):339--372,
  2001.

\bibitem{yang-shimura}
Tonghai Yang.
\newblock Rational structure of {$X(N)$} over {$\Bbb{Q}$} and explicit {G}alois
  action on {CM} points.
\newblock {\em Chin. Ann. Math. Ser. B}, 37(6):821--832, 2016.

\end{thebibliography}

\end{document}